\pdfoutput=1
\RequirePackage{ifpdf}
\ifpdf % We~are running pdfTeX in pdf mode
\documentclass[pdftex]{sigma}%,draft
\else
\documentclass{sigma}
\fi

\usepackage{multirow}
\usepackage{mathtools}
\usepackage{tikz-cd}
\usepackage{stackrel}

\newtheorem{Theorem}{Theorem}[section]
\newtheorem*{Theorem*}{Theorem}
\newtheorem{Corollary}[Theorem]{Corollary}
\newtheorem{Lemma}[Theorem]{Lemma}
\newtheorem{Proposition}[Theorem]{Proposition}
 { \theoremstyle{definition}

\newtheorem{Example}[Theorem]{Example}
\newtheorem{Remark}[Theorem]{Remark}

}

\newcommand{\SN}{\mathbb{N}} % Natural numbers
\newcommand{\SC}{\mathbb{C}} % Complex numbers
\newcommand{\C}{\SC} % Complex numbers
\newcommand{\End}{\operatorname{End}}
\newcommand{\GL}{\operatorname{GL}}
\newcommand{\head}{\operatorname{h}}
\newcommand{\Hom}{\operatorname{Hom}}

\newcommand{\Rep}{\operatorname{Rep}}
\newcommand{\SL}{\operatorname{SL}}

\newcommand{\tail}{\operatorname{t}}

\newcommand{\PiGamma}{B}
\newcommand{\PiQ}{A}

\DeclareMathOperator{\module}{\mathrm{-mod}}
\DeclareMathOperator{\Hilb}{\mathrm{Hilb}}
\DeclareMathOperator{\Quot}{\mathrm{Quot}}

\newcommand{\QuotI}[1][ ]{\Quot_I^{#1}}
\newcommand{\QuotInI}{\QuotI[n_I]}
\newcommand{\Quoti}[1][ ]{\Quot_i^{#1}}
\newcommand{\Quotini}{\Quoti[n_i]}
\newcommand{\vprime}{v'}

\numberwithin{equation}{section}

\mathchardef\mhyphen="2D
\newcommand{\ngammahilb}{n\Gamma\mhyphen\Hilb\big(\SC^2\big)}

\begin{document}
\allowdisplaybreaks

\newcommand{\arXivNumber}{2106.10115}

\renewcommand{\thefootnote}{}

\renewcommand{\PaperNumber}{099}

\FirstPageHeading

\ShortArticleName{Quot Schemes for Kleinian Orbifolds}

\ArticleName{Quot Schemes for Kleinian Orbifolds\footnote{This paper is a~contribution to the Special Issue on Enumerative and Gauge-Theoretic Invariants in honor of Lothar G\"ottsche on the occasion of his 60th birthday. The~full collection is available at \href{https://www.emis.de/journals/SIGMA/Gottsche.html}{https://www.emis.de/journals/SIGMA/Gottsche.html}}}

\Author{Alastair CRAW~$^{\rm a}$, S{\o}ren GAMMELGAARD~$^{\rm b}$, \'Ad\'am GYENGE~$^{\rm c}$ and Bal\'{a}zs SZENDR\H{O}I~$^{\rm b}$}

\AuthorNameForHeading{A.~Craw, S.~Gammelgaard, \'A.~Gyenge and B.~Szendr\H{o}i}

\Address{$^{\rm a}$~Department of Mathematical Sciences, University of Bath,\\
\hphantom{$^{\rm a}$}~Claverton Down, Bath BA2 7AY, UK}
\EmailD{\href{mailto:a.craw@bath.ac.uk}{a.craw@bath.ac.uk}}

\Address{$^{\rm b}$~Mathematical Institute, University of Oxford, Oxford OX2 6GG, UK}
\EmailD{\href{mailto:gammelgaard@maths.ox.ac.uk}{gammelgaard@maths.ox.ac.uk}, \href{mailto:szendroi@maths.ox.ac.uk}{szendroi@maths.ox.ac.uk}}

\Address{$^{\rm c}$~Alfr\'ed R\'enyi Institute of Mathematics, Re\'altanoda utca 13-15, 1053, Budapest, Hungary}
\EmailD{\href{mailto:Gyenge.Adam@renyi.hu}{Gyenge.Adam@renyi.hu}}

\ArticleDates{Received June 29, 2021, in final form November 03, 2021; Published online November 10, 2021}

\Abstract{For a finite subgroup $\Gamma\subset {\mathrm{SL}}(2,\SC)$, we identify fine moduli spaces of certain cornered quiver algebras, defined in earlier work, with orbifold Quot schemes for the Kleinian orbifold $\big[\SC^2\!/\Gamma\big]$. We also describe the reduced schemes underlying these Quot schemes as Nakajima quiver varieties for the framed McKay quiver of~$\Gamma$, taken at specific non-generic stability parameters. These schemes are therefore irreducible, normal and admit symplectic resolutions. Our results generalise our work~[\textit{Algebr.\ Geom.}~\textbf{8} (2021), 680--704] on the Hilbert scheme of points on $\C^2/\Gamma$; we present arguments that completely bypass the ADE classification.}

\Classification{16G20; 13A50; 14E16}

\renewcommand{\thefootnote}{\arabic{footnote}}
\setcounter{footnote}{0}

\section{Introduction}\label{sec:intro}

Let $\Gamma\subset{\mathrm{SL}}(2,\SC)$ be a finite subgroup and let $\Hilb^n\big(\SC^2/\Gamma\big)$ denote the Hilbert scheme of $n$ points on the singular surface $\SC^2/\Gamma$, parametrising
ideals of colength $n$ in the coordinate ring $\SC[x,y]^\Gamma$ of $\SC^2/\Gamma$. In an earlier paper~\cite{craw2019punctual}, we proved that (the reduced scheme underlying) this scheme is irreducible, normal, and admits a unique projective symplectic resolution
\[ \pi_n\colon \ \ngammahilb \longrightarrow \Hilb^n\big(\SC^2/\Gamma\big)_{\rm red},
\]
where $\ngammahilb$ is a certain component of the $\Gamma$-fixed locus
$\Hilb\big(\SC^2\big)^\Gamma$ on the Hilbert scheme of points on
$\SC^2$ itself. It was well known that $\ngammahilb$ is a quiver variety $\mathfrak{M}_{\theta}(1,n\delta)$ associated to the framed McKay quiver of $\Gamma$, where $\theta$ is generic in a GIT chamber $C_{n \delta}^+$ and where the dimension vector on the McKay quiver is taken as a multiple of a fixed vector~$\delta$.
In \cite[Theorem~1.1]{craw2019punctual} we reconstructed
the morphism $\pi_n$ by variation of GIT quotient for quiver varieties; specifically, we identified a parameter $\theta_0$ in the boundary of the chamber $C_{n \delta}^+$ such that $\pi_n$ is the morphism from $\mathfrak{M}_{\theta}(1,n\delta)$ to $\mathfrak{M}_{\theta_0}(1,n\delta)$. This allowed us to interpret by variation of GIT quotient all possible ways in which $\pi_n$ can be factored as a sequence of primitive contractions.

To prove these results in~\cite{craw2019punctual}, we introduced
a family of algebras $\PiQ_I$ obtained by `cornering', each
indexed by a non-empty subset $I$ of the set of irreducible
representations of $\Gamma$. We showed
that the fine moduli spaces of modules $\mathcal{M}_{\PiQ_I}(1,n\delta_I)$
for these algebras with very specific dimension vectors
are isomorphic to the quiver varieties mentioned above.
Moreover, for $I$ consisting of the trivial
representation only, we recovered the space $\Hilb^n\big(\SC^2/\Gamma\big)$.

In this paper, we have three main goals.
First, we aim to
understand better the moduli spaces $\mathcal{M}_{\PiQ_I}$,
constructing them as a kind of Quot scheme
that is both natural and geometric.
Second, with applications in mind, we consider arbitrary dimension vectors, not just multiples of some set of fixed dimension vectors as above. One interpretation of our results is a new fine moduli space structure, as well as a (noncommutative)
geometric interpretation, of a large class of
Nakajima quiver varieties for certain non-generic stability parameters.
Our third and final goal is to provide proofs that completely bypass any case-by-case analysis of ADE diagrams.

To state our main result,
let $r$ denote the number of nontrivial irreducible representations of~$\Gamma$.
For any non-empty subset $I\subseteq \{0,1,\dots,r\}$, let $\mathfrak{M}_{\theta_I}(1,v)$ denote the Nakajima quiver variety associated to the affine ADE graph for some dimension vector~$(1,v)$, where~$\theta_I$ is a~specific stability condition determined by our choice of~$I$; see~\eqref{eqn:theta_I} for the definition.

\begin{Theorem}\label{thm:mainintro}\label{mainthm}
Let $I\subseteq \{0,\ldots, r\}$ be a non-empty subset and
let $n_I\in {\mathbb N}^I$ be a vector. There is an orbifold Quot scheme $\QuotInI \big(\big[\SC^2/\Gamma\big]\big)$ such that:
\begin{enumerate}\itemsep=0pt
\item[$1)$] there is an isomorphism $\QuotInI \big(\big[\SC^2/\Gamma\big]\big) \cong \mathcal{M}_{\PiQ_I} (1,n_I)$ to a fine moduli space associated to the algebra $A_I$ obtained by cornering;
\item[$2)$] when $I$ is a singleton corresponding to a one-dimensional representation of $\Gamma$, the orbifold Quot scheme is isomorphic to a
classical Quot scheme for $\C^2/\Gamma$;
\item[$3)$]
the orbifold Quot scheme is non-empty if and only if the quiver variety $\mathfrak{M}_{\theta_I}(1,v)$ is non-empty for some vector
$v\in \mathbb{N}^{r+1}$ satisfying $v_i=n_i$ for all $i\in I$, in which case, after changing the values of $v_k$ for $k\not\in I$ if necessary, we have 
\[
\QuotInI \big(\big[\SC^2/\Gamma\big]\big)_{\rm red} \cong \mathfrak{M}_{\theta_I}(1,v),
\]
 where on the left we take $\QuotInI \big(\big[\SC^2/\Gamma\big]\big)$ with the reduced scheme structure.
\end{enumerate}
In particular, when it is non-empty, $\QuotInI \big(\big[\SC^2/\Gamma\big]\big)_{\rm red}$ is irreducible, normal, and it has symplectic, hence rational Gorenstein, singularities. Moreover, it admits at least one projective symplectic resolution.
\end{Theorem}

 \begin{Remark}
 Fix the dimension vector $v=n\delta$ that we refer to in our opening paragraph.
 \begin{enumerate}\itemsep=0pt
 \item For $I=\{0\}$, the statement of
 Theorem~\ref{mainthm} specialises to \cite[Theorem~1.1]{craw2019punctual}.
 Indeed, the classical Quot scheme from Theorem~\ref{mainthm}(2) specialises to the Hilbert scheme in this case by Example~\ref{rem:crsstck}, and the isomorphism from Theorem~\ref{mainthm}(3) constructs $\Hilb^{[n]}\big(\mathbb{C}^2/\Gamma\big)_{\rm red}$ as a quiver variety by Example~\ref{ex:ndeltaVnI} and Theorem~\ref{thm:mthetamorph1}. Moreover, the polarising ample bundle on $\mathfrak{M}_{\theta_I}(1,v)$ lies in an extremal ray of the movable cone described in \cite{bellamy2020birational}, so it admits a unique projective symplectic resolution.
 In fact, our arguments here reprove the main result of~\cite{craw2019punctual} without having to resort to any case-by-case analysis of ADE diagrams.
 \item For $I=\{1,\dots, r\}$, the stability condition $\theta_I$ lies in the relative interior of a wall of the chamber $C^+_{n\delta}$. It follows from \cite[Theorem~1.2]{bellamy2020birational} that the polarising ample bundle on $\mathfrak{M}_{\theta_I}(1,v)$ lies in the interior of the movable cone, so $\mathfrak{M}_{\theta_I}(1,v)$ admits more than one (in fact, precisely two) projective crepant resolutions.
 \end{enumerate}
 Thus, while the projective symplectic resolution is unique in \cite[Theorem~1.1]{craw2019punctual}, it is not possible to assert uniqueness
 in Theorem~\ref{thm:mainintro} above.
 \end{Remark}

While some of our
arguments follow~\cite{craw2019punctual} closely,
we introduce
Quot schemes in a general noncommutative
context, establishing a basic representability result in Proposition~\ref{prop:Qnrep} that may be of broader interest. Furthermore, while the
morphism
\[ \mathfrak{M}_\theta(1,v) \longrightarrow \mathfrak{M}_{\theta_I}(1,v) \]
induced by varying a generic stability condition $\theta$ to a special one $\theta_I$ is always surjective for the
dimension vector $v=n\delta$ relevant to the case $\Hilb^n\big(\SC^2/\Gamma\big)$ discussed above,
the same is not true in general.
Instead, one has to correct $v$ to
obtain a surjective morphism of quiver varieties.

The main application of
this paper will be in studying
generating functions attached to these moduli spaces,
following~\cite{gyenge2017euler, gyenge2018euler}.
In forthcoming work, various generating functions of Euler numbers of Hilbert
and Quot schemes for ADE singularities and orbifolds will be studied, demonstrating
specialisation phenomena also on the level of these generating series.
A poset of Quot schemes corresponding to the poset of
subsets of $\{0, \dots, r\}$,
generalising that of \cite[Section~5]{craw2019punctual},
will play an important role in that study.
As in~\cite{gyenge2017euler, gyenge2018euler}, it will be essential
that we work with arbitrary dimension vectors,
not just multiples of $\delta$.

The structure of the paper is as follows. In Section~\ref{sec:prelim} we collect
the necessary preliminaries on semi-invariants of finite group actions and the McKay correspondence. In Section~\ref{sec:quot}, we define Quot schemes for Kleinian orbifolds and discuss some of their properties. In Section~\ref{subsec:finemod}, we introduce the algebras $A_I$
and the associated fine moduli spaces, and we identify these moduli spaces with Quot schemes for Kleinian orbifolds. In Section~\ref{sec:quivers},
we obtain a resolution of singularities for the quiver varieties of interest.
Finally, in Section~\ref{sec:quivfine} we establish the key isomorphism from Theorem~\ref{thm:mainintro}(3).

\medskip

\noindent
\textbf{Conventions.}
 We work throughout over $\SC$. In particular, all schemes are $\SC$-schemes and all tensor products are taken over $\SC$ unless otherwise indicated. We often use the following partial order on dimension vectors: for any $n\in \mathbb{N}$ and for $u=(u_i),v=(v_i)\in \mathbb{Z}^n$, we define $u\le v$ if $u_i\le v_i$ for all $1\leq i\leq n$. We adopt the sign convention of King~\cite{king1994moduli} for $\theta$-stability, so a~module $M$ over an algebra is $\theta$-stable (resp.\ $\theta$-semistable) if $\theta(M)=0$ and every nonzero proper submodule $N\subsetneq M$ satisfies $\theta(N)>0$ (resp.\ $\theta(N)\ge 0$).

\section{Preliminaries}\label{sec:prelim}

\subsection{Three algebras arising from the McKay graph}\label{subsec:mckay}
Let $\Gamma\subset \GL(V)$ be a finite subgroup of the general linear group for an $m$-dimensional $\SC$-vector space $V$.
Choosing a basis of $V$, we can write the symmetric algebra of the dual vector space $V^*$ as a polynomial ring $R=\SC[x_1,\dots,x_m]$. The group $\Gamma$ acts dually on $R$: for $g\in \Gamma$ and $f\in R$, we have $(g\cdot f)(v) = f\big(g^{-1} v\big)$ for all $v\in V$. Let $R^\Gamma$ denote the $\Gamma$-invariant subring of~$R$.

List the irreducible representations of $\Gamma$ as $\{\rho_0, \rho_1, \dots, \rho_r\}$, where $\rho_0$ is the trivial representation.
The decomposition of $R$ as an $R^\Gamma$-module can be written \cite[Proposition~4.1.15]{goodman2009symmetry}
 \begin{gather*} %\label{eqn:decompR}
 R\cong \bigoplus_{0\leq i\leq r} R_i\otimes_\mathbb{C} \rho_i,
 \end{gather*}
 where $R_i\coloneqq \Hom_\Gamma(\rho_i,R)$ is an
 $R^\Gamma$-module; note that $R_0=R^{\Gamma}$.

From now on, let $\Gamma \subset \SL(2,\SC)$ be a finite subgroup acting as above on $R=\SC[x,y]$. The \emph{McKay graph} of $\Gamma$
has vertex set $\{0,1, \dots, r\}$ corresponding to the irreducible representations of~$\Gamma$, where for each $0\leq i, j\leq r$, we have $\dim\Hom_{\Gamma}(\rho_j , \rho_i \otimes V )$ edges joining vertices~$i$ and~$j$. McKay~\cite{McKay80} observed that this graph is an affine Dynkin diagram of ADE type. We frame this graph by introducing an additional vertex, denoted~$\infty$, together with an edge joining the vertices~$\infty$ and~$0$; we refer to the resulting graph as the \emph{framed McKay graph} of~$\Gamma$.

We now associate a doubled quiver to each graph. For this, consider
the set of pairs $Q_1$ comprising an edge in the framed McKay graph
and an orientation of the edge. For each $a\in Q_1$, we write~$t(a)$,
$h(a)$ for the vertices at the tail and head respectively of
the oriented edge, and we write $a^{\ast}$ for the same edge with
the opposite orientation. Define the \emph{framed McKay quiver of $\Gamma$},
denoted $Q$, to be the quiver with vertex set $Q_0= \{\infty,0,\dots ,r\}$
and arrow set $Q_1$. The (unframed) \emph{McKay quiver},
denoted $Q_\Gamma$, is the complete subquiver of $Q$ on
the vertex set $\{0, 1,\dots, r\}$. Using these quivers
we define the following three algebras.

First, let $\SC Q$ denote the path algebra of the framed McKay quiver $Q$. For $i \in Q_0$, let $e_i \in \SC Q$ denote the idempotent corresponding to the trivial path at vertex $i$. Let $\epsilon\colon Q_1 \to \{\pm 1\}$ be any map such that $\epsilon(a) \neq \epsilon(a^{\ast})$ for all $a \in Q_1$. The \emph{preprojective algebra of $Q$}, denoted $\Pi$, is defined as the quotient of $\SC Q$ by the ideal
\begin{equation}
\label{eqn:preprojectiverelation}
\left\langle \sum_{\head(a)=i} \epsilon(a) aa^{\ast} \mid {i \in Q_0}
\right\rangle.
\end{equation}
 The preprojective algebra $\Pi$ does not depend on the choice of the map $\epsilon$. For each $i\in Q_0$, we simply write $e_i\in \Pi$ for the image of the corresponding vertex idempotent.

Second, in the framed McKay quiver $Q$, let $b^*\in Q_1$ be the unique arrow with head at vertex~$\infty$. Define a $\SC$-algebra as the quotient of the preprojective algebra $\Pi$ by the two-sided ideal generated by the class of $b^*$:
 \[
 \PiQ\coloneqq \Pi/(b^*).
 \]
 Equivalently, if we define a quiver $Q^*$ to have vertex set $Q^*_0=\{\infty,0,1,\dots,r\}$ and arrow set $Q_1^*=Q_1\setminus \{b^*\}$, then $\PiQ$ is the quotient of the path algebra $\SC Q^*$ by the ideal of relations
 \begin{equation*}% \label{eqn:Arelations}
 \left\langle \sum_{\head(a) = i} \epsilon(a) aa^* \mid 0\leq i,
 \head(a^*)\leq r \right\rangle.
 \end{equation*}

 Third, let $\SC Q_\Gamma$ denote the path algebra of the (unframed) McKay quiver $Q_\Gamma$. Since $\SC Q_\Gamma$ is a subalgebra of $\SC Q$, we use the same symbol $e_i\in \SC Q_\Gamma$ for the idempotent corresponding to the trivial path at vertex $i$ of $Q_\Gamma$. The \emph{preprojective algebra of $Q_\Gamma$}, denoted $\PiGamma$, is defined as the quotient of $\SC Q_\Gamma$ by an ideal defined similarly to that from~\eqref{eqn:preprojectiverelation} for the quiver $Q_\Gamma$.

\begin{Remark}
 For each vertex $0\leq i\leq r$ we denote by $e_i$ the class in $\PiQ$ of the idempotent at the vertex $i$. Note that $\PiGamma$ is not a subalgebra of $\Pi$. On the other hand, $\PiGamma$ is isomorphic to the subalgebra $\oplus_{0 \leq i,j \leq r} e_j\PiQ e_i$ of $\PiQ$ due to \cite[Lemma~3.2]{craw2019punctual}, and moreover
 \begin{equation} \label{eq:PiQdecomp}
 \PiQ \cong \PiGamma \oplus \PiGamma b \oplus \SC e_\infty.
\end{equation}
as complex vector spaces. In particular, the classes $e_i$, $0 \leq i \leq r$ are also contained in $\PiGamma$.
\end{Remark}

\subsection{Morita equivalence}
The skew group algebra $S \coloneqq R \ast \Gamma=\SC[x,y]\ast \Gamma$ of the finite subgroup $\Gamma\subset \SL(2,\SC)$ contains the group algebra $\SC\Gamma$ as a subalgebra. For $0\leq i\leq r$, choose an idempotent $f_i\in \mathbb{C}\Gamma$ such that $\mathbb{C}\Gamma f_i\cong \rho_i$.

\begin{Lemma}\label{lem:polyring}
There is an isomorphism $S f_0 \cong R$ of $S\text{-}R^\Gamma$-bimodules.
\end{Lemma}
\begin{proof}
As in \cite[Lemma~1.1]{crawley1998noncommutative}, $S$ has a basis given by all elements of the form $x^jy^kg$, where $j,k \geq 0$ and $g \in \Gamma$. As $f_0$ is invariant under multiplication by any $g \in \Gamma$, the subspace $Sf_0$ has a basis given by elements of the form $x^jy^kf_0$ for $j,k \geq 0$. This gives the required isomorphism
of vector spaces which is readily seen to be an isomorphism of $S\text{-}R^\Gamma$-bimodules.
\end{proof}

To recall the key result, define $f=f_0+\dots+f_r$.

\begin{Proposition}\label{prop_morita}
The skew group algebra $S$ and the preprojective algebra $B$ are Morita equivalent via an isomorphism $f S f\cong B$; explicitly, the equivalence on left modules is
\begin{align*}\Phi\colon \ S\module & \longrightarrow B\module, \\
M & \mapsto fM.
\end{align*}
Moreover, there are $\mathbb{C}$-algebra isomorphisms
$S\cong \End_{R_0}\big(R)$ and $B\cong \End_{R_0}\big({\bigoplus}_{0\leq i\leq r} R_i\big)$.
\end{Proposition}

\begin{proof} The Morita equivalence is a bi-product of the proof of \cite[Proposition~2.13]{reiten1989two}, and may have been first stated explicitly in~\cite[Theorem~0.1]{crawley1998noncommutative}. The algebra isomorphism $f S f\cong B$ that in particular sends $f_i$ to $e_i$ for $0\leq i\leq r$ is described in~\cite[Theorem~3.4]{crawley1998noncommutative}. The description of~$S$ as $\End_{R_0}(R)$ follows from~\cite{Auslander62}, spelled out in
\cite[Theorem~5.12]{LeuschkeWiegand12}, while the description of~$B$ follows by applying the isomorphism $fSf\cong B$; see for example Buchweitz~\cite{Buchweitz}.
\end{proof}

\begin{Corollary}
\label{cor:moritaRiBi}
 For any $0\leq i\leq r$, there is an isomorphism $f_i S f_i\cong e_iBe_i$ of $\SC$-algebras, as well as an isomorphism $R_i \cong e_iBe_0$ of $e_iBe_i\text{-}e_0Be_0$-bimodules.
\end{Corollary}
\begin{proof}
Apply $f_i$ to both sides of $fSf$ while simultaneously applying $e_i$ to both sides of $B$; the algebra isomorphism $fSf\cong B$ then restricts to the algebra isomorphism $f_i S f_i\cong e_iBe_i$. For the second statement, apply $e_i$ and $e_0$ to the left and right respectively of the isomorphism $B\cong \End_{R_0}\big({\bigoplus}_{0\leq i\leq r} R_i\big)$ to obtain $R_i\cong \Hom_{R_0}(R_0,R_i)\cong e_i B e_0$ as required.
\end{proof}

We finally look at how the Morita equivalence $\Phi$ acts on dimension vectors.

\begin{Corollary}\label{cor_fd}
Given an $S$-module of the form $R/J$, with dimension vector $\sum_{0\leq i\leq r} v_i \rho_i$, its image under $\Phi$ is a $B$-module of the form $B e_0/fJ$ with dimension vector $(v_i)\in\mathbb{Z}^{r+1}$.
\end{Corollary}
\begin{proof}
Applying Lemma~\ref{lem:polyring} shows that $\Phi$ sends the polynomial ring $R\cong Sf_0$ to $fSf_0 =(fSf)f_0 = Be_0$. It follows that for any
$S$-submodule
$J\subseteq R$, the quotient $R/J$ is an $S$-module whose image under $\Phi$ is $Be_0/fJ$. Finally, $\Phi$ induces a $\mathbb{Z}$-linear isomorphism between the Grothendieck groups of the categories of finite-dimensional left modules over $S$ and $B$ which we identify with the representation ring $\Rep(\Gamma)$ and with $\mathbb{Z}^{r+1}$ respectively.
\end{proof}

\section{Orbifold Quot schemes}\label{sec:quot}

\subsection{Quot schemes for modules over associative algebras}\label{sect:generalquot}

Let~$H$ denote an arbitrary finitely generated,
not necessarily commutative $\SC$-algebra, and let~$M$ be a finitely
generated left $H$-module.
For a scheme $X$, we denote by $H_{X}$ the sheaf of $\mathcal{O}_X$-algebras
$H \otimes \mathcal{O}_X$, where $H$ is considered as a constant sheaf on $X$. Similarly, denote by $M_{X}$ the sheaf of $H_{X}$-modules
$M \otimes \mathcal{O}_X$.

Consider the functor
\[
\mathcal{Q}^{n}_H(M)\colon \ \mathrm{Sch}^{\mathrm{op}} \to \mathrm{Sets} \]
sending a scheme $X$ to the set of isomorphism classes of left $H_{X}$-modules $Z$ equipped with a~surjective $H_{X}$-module homomorphism
$M_{X} \to Z$ such that $Z$,
when considered as an $\mathcal{O}_{X}$-module, is locally free of rank~$n$.

The following result may be known to experts, though we could not find it in the literature.

\begin{Proposition}
\label{prop:Qnrep}
The functor $\mathcal{Q}^{n}_H(M)$ is represented by a scheme $\Quot^{n}_H(M)$
of finite type over~$\SC$.
\end{Proposition}
\begin{proof}
As $H$ is finitely generated as a $\SC$-algebra, we can fix a surjection $H^+\to H$ with $H^+$ a free noncommutative $\SC$-algebra on a finite number of generators.
Then $M$ is also a left $H^+$-module. Define $J\coloneqq \ker\big(H^+ \to H\big)$. Consider a left $H^+$-module $Z$ that is a quotient of~$M$ via a surjective $H^+$-module morphism $M \to Z$.

Since any element of the two-sided ideal $J$ of $H^+$ acts trivially on $M$, and so also on quotients of $M$, $Z$ is automatically an $H$-module. Conversely, any quotient of $M$ as an $H$-module is automatically an $H^{+}$-module. It follows that there is a canonical isomorphism of functors $\mathcal{Q}^n_H(M) = \mathcal Q^n_{H^+}(M)$.

Thus, we may assume that $H$ is a finitely generated and free $\C$-algebra.
Fix a left $H$-module surjection $\psi\colon H^r \to M$ for some $r$.
 Write elements of $\ker \psi$ in the form $\sum_{1\leq k \leq r} w_{k} e_k $, where the $e_k \in H^r$ are the standard module generators and $w_k \in H$.
 For any quotient module $q\colon M \to Z$, composition with $\psi$ presents $Z$ as a quotient $\tilde q\colon H^r \to Z$,
 such that $\ker \psi\subseteq\ker\tilde q$.
 The equations $\sum_{1\leq k \leq r} w_{k} \tilde q(e_k)=0$ give
 a closed condition on $[\tilde q]\in \Quot^n_H(H^r)$: the vanishing of a collection of vectors in the vector space underlying~$Z$.
 When trivialising the universal sheaf on some open cover of $Q^n_H(H^r)$, these closed conditions glue together, and
 we see that this construction realises
\[\mathcal{Q}^n_H(M) \subseteq \mathcal{Q}^n_H(H^r)\]
as a closed subfunctor.

To conclude, it suffices to show that $\mathcal{Q}^n_H(H^r)$ is representable for a free noncommutative $\SC$-algebra $H$. This follows from~\cite[Theorem 2.5]{beentjes2018virtual}; while that paper discusses the case when $H$ is freely generated by $3$ elements, the proof generalises to finitely many generators.
\end{proof}

\subsection{Quot schemes for Kleinian orbifolds}
\label{subsec:Quotorbi}
Let $\PiGamma$ be the preprojective algebra of the unframed McKay quiver and $R=\SC[x,y]$ as
in Section~\ref{subsec:mckay}.
Recall that $R_0=R^\Gamma$, and that for each $0\leq i \leq r$, the
$R_0$-module $R_i\coloneqq\Hom_\Gamma(\rho_i,R)$ satisfies $R_i\cong e_iBe_0$.

For any subset $I \subseteq \{0,\dots,r\}$,
consider the idempotent
$e_I\coloneqq\sum_{i \in I} e_i$ in $B$,
and define the $\SC$-algebra
\begin{equation} \label{eqn:BI}
\PiGamma_I\coloneqq e_I \PiGamma e_I \end{equation}
comprising linear combinations of
classes of paths in $Q_\Gamma$ whose tails and heads lie in the set~$I$. The process of passing
from $\PiGamma$ to $\PiGamma_I$ is called \emph{cornering}; see \cite[Remark 3.1]{craw2018multigraded}.
Then \[ R_I\coloneqq\bigoplus_{i \in I} R_i \cong e_I \PiGamma e_0\] is naturally a finitely generated left $\PiGamma_I$-module.

For a given dimension vector $n_I=(n_i) \in \mathbb{N}^{I}$,
consider the contravariant functor
\[
\mathcal{Q}^{n_I}_{\PiGamma_I}(R_I)\colon \ \mathrm{Sch}^{\mathrm{op}} \to \mathrm{Sets} \]
sending a scheme $X$ to the set of isomorphism classes of left $\PiGamma_{I,X}$-modules $Z$ equipped with a~surjective $\PiGamma_{I,X}$-module homomorphism
\[
\phi_Z\colon \ R_I
\otimes \mathcal{O}_X \to Z
\]
such that each submodule $e_iZ$, {for $i\in I$,}
when considered as an $\mathcal O_{X}$-module, is locally free of rank $n_i$. Note that we can recover $Z$ from these submodules via
\begin{equation*}
Z=\bigoplus_{i\in I} e_i Z.
\end{equation*}

The next result provides the link between this functor and that introduced in Section~\ref{sect:generalquot}.
\begin{Proposition} There is a finite decomposition
\begin{equation}\mathcal{Q}^{n}_{\PiGamma_I}(R_I)=\coprod_{\substack{n_I\\ \sum_{i \in I} n_i =n}} \mathcal{Q}^{n_I}_{\PiGamma_I}(R_I)
\label{eq:funct_decomp}\end{equation}
into open and closed subfunctors.
In particular, each functor
$\mathcal{Q}^{n_I}_{\PiGamma_I}\big(R_I\big)$ is represented by a scheme $\QuotInI\big(\big[\SC^2/\Gamma\big]\big)$ of finite type over $\SC$.
\end{Proposition}
\begin{proof} The dimension vector $n_I=(n_i) \in \mathbb{N}^{I}$ is locally constant in a flat family of $\PiGamma_I$-modules of fixed rank~$n$, with each of the entries being lower semicontinuous by the Schur lemma.
So we have a decomposition~\eqref{eq:funct_decomp} into open and closed subfunctors.
As $R_I$ is finitely generated as a left $B_I$-module, the functor on the left hand side of~\eqref{eq:funct_decomp} is represented by a scheme of finite type over $\SC$ by Proposition~\ref{prop:Qnrep}. The last claim then follows.
\end{proof}

When the index set is a singleton $I=\{i\}$, we
often simply write $\Quotini \big(\big[\SC^2/\Gamma\big]\big)$ for $\Quot_{\{i\}}^{n_{\{i\}}}\big(\big[\SC^2/\Gamma\big]\big)$. See Section~\ref{sect_quot_quot} for a
discussion of this special case.

As is common with Hilbert schemes in other contexts, we
consider collections of our Quot schemes for all
possible dimension vectors. Define the
orbifold Quot scheme for $\big[\SC^2/\Gamma\big]$ and for the index set $I$ to be
\[ \QuotI \big(\big[\SC^2/\Gamma\big]\big)\coloneqq\coprod_{n_I} \QuotInI\big(\big[\SC^2/\Gamma\big]\big) \]
where $n_I=(n_i)_{i \in I} \in \mathbb{N}^{I}$ as before.
Again, when $I=\{i\}$ is a singleton, we simply write
\[
\Quoti \big(\big[\SC^2/\Gamma\big]\big)\coloneqq\coprod_{n_i\in \mathbb{N}}\Quotini \big(\big[\SC^2/\Gamma\big]\big).\]

 Denote by $\Hilb\big(\big[\SC^2/\Gamma\big]\big)$ the Hilbert scheme of $\Gamma$-invariant finite colength subschemes of~$\SC^2$. As explained in more detail in \cite[Section~1.1]{gyenge2018euler}, this space decomposes as a disjoint union of quasi-projective varieties
\[\Hilb\big(\big[\SC^2/\Gamma\big]\big) = \coprod_{v\in \mathbb{N}^{r+1}}\Hilb^v\big(\big[\SC^2/\Gamma\big]\big),
\]
where \[\Hilb^v\big([\SC^2/\Gamma]\big) =\Big\{J\in \Hilb\big(\big[\SC^2/\Gamma\big]\big)\mid H^0\big(\mathcal{O}_{\SC^2}/J\big) \cong \bigoplus_{i\in \{0, \dots, r\}} \rho_i^{\oplus v_i}\Big\}.\]

\begin{Lemma}\label{lem:HilbMor}
For $I=\{0,\dots,r\}$, there is an isomorphism
\[\QuotI \big(\big[\SC^2/\Gamma\big]\big) \cong \Hilb\big(\big[\SC^2/\Gamma\big]\big)\]
which respects the decompositions on both sides into pieces indexed by $v\in \mathbb{N}^{r+1}$.
\end{Lemma}
\begin{proof} This follows from the Morita equivalence between the algebras $S$
and $\PiGamma$ from Proposition~\ref{prop_morita}, and
its amplification Corollary~\ref{cor_fd}, together with the fact that a finite colength left $S$-submodule of $R=\SC[x,y]$ is nothing but a finite colength $\Gamma$-invariant ideal of $\SC[x,y]$.
\end{proof}

\subsection{Quot schemes for Kleinian singularities}
\label{sect_quot_quot}
In this section, we discuss the relation of the construction from Section~\ref{subsec:Quotorbi} to the more classical (commutative) version of the Quot scheme, in the case when $I=\{i\}$ is a singleton.
Note that in this case, $R_i\cong e_i\PiGamma e_0$
 is both a left $B_i$-module and a right $R_0\cong e_0\PiGamma e_0$-module.

Let $\Quoti \big(\SC^2/\Gamma\big)$ be the scheme parameterising finite codimension $R^{\Gamma}$-submodules of $R_i$
or, equivalently, finite colength subsheaves of the coherent sheaf on the Kleinian singularity $\SC^2/\Gamma$ corresponding to $R_i$.
Again, this space decomposes as
\[ \Quoti \big(\SC^2/\Gamma\big)=\coprod_{n_i\in \mathbb{N}} \Quotini \big(\SC^2/\Gamma\big) \]
with $\Quotini \big(\SC^2/\Gamma\big) = \big\{M \in \Quoti \big(\SC^2/\Gamma\big) \mid
\mathrm{dim}\,R_i/M =n_i
\big\}$.

Suppose that $i\in\{0,\ldots, r\}$ satisfies $\dim \rho_i =1$. Then there is a diagram automorphism of $Q_\Gamma$ taking vertex~$i$ to vertex~$0$. Pick any such
automorphism $\iota$ and let
$\iota(0) \in \{0,\dots,r\}$ denote the vertex to which the
vertex $0$ is mapped under this diagram automorphism. This transformation of the diagram induces an algebra isomorphism \[B_i=e_i \PiGamma e_i\cong e_0 \PiGamma e_0=B_0.\]
Also, it induces a bijection between left $B_i$-submodules of $e_iBe_0$ and left $B_0$-submodules of $e_{0}Be_{\iota(0)}$. Reversing the arrows gives an algebra automorphism $B_0 \xrightarrow{\sim} B_0$ as well as a bijective correspondence between left $B_0$-submodules of $e_{0}Be_{\iota(0)}$ and right $B_0$-submodules of $e_{\iota(0)}Be_{0}$. Putting these two correspondences together, we get a bijection between the left $B_i$-submodules of $R_i=e_iBe_0$ and right $B_0$-submodules of $R_{\iota(0)}=e_{\iota(0)}Be_{0}$.

\begin{Proposition}\label{prop:Quotcommiso}
Suppose that $I=\{i\}$ is a singleton such
that $\dim \rho_i =1$. Let $n_i \in \mathbb{N}$
be a~dimension vector for the index set $\{i\}$, and denote
by $n_{\iota(0)}$ the same dimension vector for the
index set $\{\iota(0)\}$ with $\iota(0)$
as above. Then $\Quotini \big(\big[\SC^2/\Gamma\big]\big)$ is isomorphic
to $\Quot_{\iota(0)}^{n_{\iota(0)}}\big(\SC^2/\Gamma\big)$.
\end{Proposition}
\begin{proof}
The above correspondence between left $B_i$-sub\-modules of $e_iBe_0$
and right $B_0$-sub\-mo\-du\-les of $e_{\iota(0)}Be_{0}$ is functorial
and preserves the colength. By definition, the scheme
$\Quot_{\iota(0)}^{n_{\iota(0)}}\!\big(\SC^2/\Gamma\big)$ represents the functor
that sends a scheme $X$ to the set of isomorphism classes of
right $\PiGamma_{0,X}$-module quotients $R_{\iota(0),X} \to Z$ which are
locally free of rank $n_{\iota(0)}$ over $\mathcal{O}_X$. The claim follows.
\end{proof}

\begin{Remark}
The assumption of Proposition~\ref{prop:Quotcommiso} is satisfied for type $A_r$ and any $0 \leq i \leq r$, type $D_r$ and $i \in \{0,1,r-1,r\}$, $E_6$ and $i \in \{0,1,6\}$, $E_7$ and $i \in \{0,7\}$ and for $E_8$ and $i=0$. Here we have used the Bourbaki convention \cite{BourbakiLie} for enumerating the vertices.
\end{Remark}

\begin{Example}
\label{rem:crsstck}
For $i=0$, the assumption of Lemma~\ref{prop:Quotcommiso} is satisfied in all types, with $\iota(0)=0$ also. Writing $n_0=n$, we obtain isomorphisms
\[ \Quot_0^{n_0}\big(\big[\SC^2/\Gamma\big]\big)\cong \Quot_0^{n_0}\big(\SC^2/\Gamma\big)\cong \Hilb^{n}\big(\SC^2/\Gamma\big). \]
\end{Example}

\section{Fine moduli spaces of modules}
\label{subsec:finemod}

\subsection{Cornering and fine moduli spaces}

Let $\Rep(\Gamma)$
denote the representation ring of $\Gamma$. Introduce a formal symbol $\rho_{\infty}$ so that $\{\rho_i \mid i \in Q_0\}$ provides a $\mathbb{Z}$-basis for $\mathbb{Z}^{Q_0}\cong \mathbb{Z}\oplus \Rep(\Gamma)$ considered as $\mathbb{Z}$-modules.

Recall that $\PiQ=\Pi/(b^{\ast})$ is the quotient of the
preprojective algebra of the framed McKay quiver of $\Gamma$ by the two sided ideal $(b^{\ast})$, where $b^\ast$ is the arrow from $0$ to $\infty$.
Let $I \subseteq \{0, \dots , r\}$ be a non-empty subset; we work under the assumption that $I\ne \varnothing$ for the rest of the paper.
Define the idempotent $\overline{e}_I \coloneqq e_{\infty} + \sum_{i \in I} e_i \in \PiQ$.
The subalgebra
\begin{equation}\label{eqn:AI}
\PiQ_I \coloneqq \overline{e}_I \PiQ \overline{e}_I
\end{equation}
of $A$ comprises linear combinations of
classes of paths in $Q$ whose tails and heads lie in the set $\{\infty\} \cup I$; this is also an instance of \emph{cornering}~\cite[Remark~3.1]{craw2018multigraded}.

Let
$n_I \coloneqq {(n_i)_{i\in I}}= \sum_{i \in I} n_i \rho_i \in \mathbb{N}^{I}$.
Then $(1,n_I)\coloneqq\rho_{\infty} +n_I \in \mathbb{N} \oplus \mathbb{N}^{I}$
is a dimension vector for $\PiQ_I$-modules, and we consider the stability condition $\eta_{I}\colon \mathbb{Z} \oplus \mathbb{Z}^{I} \to \mathbb{Q}$
given by
\begin{equation}
 \label{eqn:etaI}
\eta_I (\rho_i) =
\begin{cases}
\displaystyle -\sum_{j \in I} n_j & \textrm{for } i= \infty, \\
1 & \textrm{if }i \in I.
\end{cases}
\end{equation}

It follows directly from the definition that an $\PiQ_I$-module $N$ of dimension vector $(1,n_I)$ is $\eta_I$-stable if and only if
there exists a surjective $\PiQ_I$-module homomorphism $\PiQ_I e_{\infty} \to N$.
The quiver moduli space construction of King~\cite[Proposition~5.3]{king1994moduli} for finite dimensional algebras can be adapted to any algebra presented as the quotient of a finite, connected quiver by an ideal of relations. We established in \cite[Proposition~3.3]{craw2019punctual} that the algebra $\PiQ_I$ has this property.
The vector~$(1,n_I)$ is indivisible and $\eta_I$ is a generic stability condition for $\PiQ_I$-modules of dimension vector~$(1,n_I)$, so there is a~fine moduli space $\mathcal{M}_{\PiQ_I}(1,n_I)$
of $\eta_I$-stable $\PiQ_I$-modules of dimension
vector~$(1,n_I)$. Let $T_I \coloneqq\bigoplus_{i \in \{\infty \}\cup I } T_i$ denote the tautological vector bundle on~$\mathcal{M}_{\PiQ_I}(1,n_I)$, where~$T_{\infty}$ is the trivial bundle and $T_i$ has rank~$n_i$ for $i \in I$.
After multiplying $\eta_I$ by a positive integer $m\geq 1$ if necessary, we may assume that the polarising ample line bundle $\mathcal{L}_I \coloneqq \bigotimes_{i \in I} \mathrm{det}(T_i)^{m}$ on $\mathcal{M}_{\PiQ_I}(1,n_I)$ given by the GIT construction is very ample.

\subsection{The Quot scheme as a fine moduli space of modules}
We now establish the link between the fine moduli space $\mathcal{M}_{\PiQ_I}(1,n_I)$ and the orbifold Quot scheme $\QuotInI\big(\big[\SC^2/\Gamma\big]\big)$.

For any non-empty subset $I \subseteq \{0,\dots,r\}$, consider the algebras $\PiQ_I$ and $\PiGamma_I$ from~\eqref{eqn:AI} and \eqref{eqn:BI} respectively.
 In \cite[Proposition~3.3]{craw2019punctual}, we used the isomorphism $\PiGamma_I\cong \End_{R_0}\big(\bigoplus_{i\in I} R_i\big)$ to introduce quivers $Q_I$ and $Q_I^*$ and a commutative diagram of $\C$-algebra homomorphisms
 \begin{equation*}
\begin{tikzcd}
 \C Q_I \ar[d,"\alpha_I"]\ar[r] & \C Q_I^* \ar[d,"\beta_I"] \\
 \PiGamma_I \ar[r] & \PiQ_I,
 \end{tikzcd}
 \end{equation*}
 where the vertical maps are surjective and the horizontal maps are injective. In particular, we obtain quiver presentations $\PiGamma_I\cong \C Q_I/\ker(\alpha_I)$ and $\PiQ_I\cong \C Q_I^*/\ker(\beta_I)$. The vertex set of $Q^*_I$ is $\{\infty\}\cup I$, while the
edge set
comprises three kinds of arrows: loops at each vertex $i\in I$; arrows between pairs of
distinct vertices in $I$; and
arrows
from vertex $\infty$ to vertices in $I$.
Note that $Q_I$ is the complete
subquiver of $Q_I^*$ on the vertex set $I$.

\begin{Lemma}\label{lem_general_I}
Equip the vector space $\PiGamma_I \oplus R_I \oplus \SC e_\infty$ with the following multiplication:
\[ (b_1,r_1,c_1e_\infty) \cdot (b_2,r_2,c_2e_\infty) \coloneqq (b_1b_2,b_1r_2+r_1c_2,c_1c_2e_\infty). \]
Then there is an isomorphism
\begin{equation*}
\PiQ_I \cong \PiGamma_I \oplus R_I \oplus \SC e_\infty
\end{equation*}
of $\SC$-algebras, where $R_I=e_I \PiGamma e_0$ is also a left $\PiGamma_I$-module.
\end{Lemma}

\begin{proof}
It follows from \eqref{eq:PiQdecomp} that there is an isomorphism of vector spaces
\begin{equation*} \PiQ_I
\cong (e_I+e_\infty) (\PiGamma \oplus \PiGamma b \oplus \SC e_\infty) (e_I+e_\infty) \cong \PiGamma_I \oplus e_I \PiGamma b \oplus \SC e_\infty.\end{equation*}
For the middle summand, consider the map of left $\PiGamma_I$-modules \begin{equation}
\label{eq:RIbij}
R_I = e_I\PiGamma e_0 \longrightarrow e_I\PiGamma b\end{equation} defined on representing paths in the framed McKay quiver $Q$ by composing with the arrow $b$ on the right. This map is by definition surjective. But it is also injective, as every relation in $\PiQ$ involving the arrow $b$ can be factored into the product of $b$ and a relation in $\PiGamma$. This establishes the isomorphism
\begin{equation*}
\PiQ_I \cong \PiGamma_I \oplus R_I \oplus \SC e_\infty
\end{equation*}
of vector spaces. To enhance this to an isomorphism of $\mathbb{C}$-algebras,
note that the algebra structure on $\PiQ_I \cong \PiGamma_I \oplus e_I \PiGamma b \oplus \SC e_\infty$ is given by
\[ (b_1,r_1b,c_1e_\infty) \cdot (b_2,r_2b,c_2e_\infty) = \big(b_1b_2,b_1r_2b+r_1bc_2e_\infty,c_1c_2e_\infty\big).\]
Since $c_2$ is a scalar, it commutes with $b=be_\infty$, giving $b_1r_2b+r_1bc_2e_\infty = (b_1r_2+r_1c_2)b$ in the middle term above. The isomorphism \eqref{eq:RIbij} allows us to drop the $b$ from the right-hand side, leaving the second statement as required.
\end{proof}

\begin{Proposition}
\label{prop:omegaiso}
For any non-empty subset $I\subseteq\{0,\dots,r\}$ and $n_I \in \mathbb{N}^{I}$, there is an isomorphism of schemes
\[\omega_{n_I} \colon \ \QuotInI\big(\big[\SC^2/\Gamma\big]\big) \longrightarrow \mathcal{M}_{\PiQ_I}(1,n_I).\]
In particular, $\QuotInI\big(\big[\SC^2/\Gamma\big]\big)$ is non-empty if and only if $\mathcal{M}_{\PiQ_I}(1,n_I)$ is non-empty.
\end{Proposition}

\begin{proof}
Denote by $\mathcal{T}$ the tautological bundle on $\QuotInI\big(\big[\SC^2/\Gamma\big]\big)$, a family of quotient $\PiGamma_I$-modules of $R_I$ of rank $\sum_{i \in I} n_i$. Moreover, let $\mathcal{O}$ be the trivial bundle on $\QuotInI\big(\big[\SC^2/\Gamma\big]\big)$.

We give $\mathcal{O}\oplus \mathcal{T}$ the structure of a family of $\PiQ_I$-modules using the decomposition from Lemma~\ref{lem_general_I}.
Indeed, the $\PiGamma_I$-module structure on $\mathcal{T}$ extends to a $\PiGamma_I$-module structure on $\mathcal{O}\oplus \mathcal{T}$ by acting by $0$ on the first summand. The summand $\SC e_\infty$ acts by scaling on $\mathcal{O}$ and trivially on $\mathcal{T}$. Finally, as the bundle $\mathcal{T}$ is a quotient of the trivial family $R_I$, any element $r$ of the middle summand $R_I \subset A_I$ determines an $\mathcal{O}$-module homomorphism $\phi_r \in \Hom(\mathcal{O},\mathcal{T})$ via the correspondence $1\mapsto [r]$. As $R_I$ is a left $\PiGamma_I$-module, the $\SC$-linear correspondence $r \mapsto \phi_r$ is left $\PiGamma_I$-linear. The action of two elements $(b_1,r_1,c_1 e_\infty), (b_2,r_2,c_2 e_\infty) \in \PiQ_I$ on $\mathcal{O}\oplus \mathcal{T}$ therefore composes as
\[ \begin{pmatrix}c_1 & 0 \\ \phi_{r_1} & b_1 \end{pmatrix} \circ \begin{pmatrix}c_2 & 0 \\ \phi_{r_2} & b_2 \end{pmatrix} = \begin{pmatrix}c_1c_2 & 0 \\ \phi_{r_1}c_2+b_1\phi_{r_2} & b_1b_2 \end{pmatrix}=\begin{pmatrix}c_1c_2 & 0 \\ \phi_{r_1c_2+b_1r_2} & b_1b_2 \end{pmatrix}.\]
Lemma~\ref{lem_general_I}
shows that
the bundle $\mathcal{O}\oplus \mathcal{T}$ on $\QuotInI\big(\big[\SC^2/\Gamma\big]\big)$
is a flat family of $\PiQ_I$-modules
of dimension vector $(1,n_I)$. Over any closed point, the resulting
$\PiQ_I$-module is generated at~$\infty$, and so it is $\eta_I$-stable.
The universal property of the fine moduli space $\mathcal{M}_{\PiQ_I}(1,n_I)$
induces the morphism of schemes~$\omega_{n_I}$.

Conversely, applying Lemma~\ref{lem_general_I}
in the opposite direction gives that to any $\eta_I$-stable $\PiQ_I$-module of dimension vector~$(1,n_I)$, there corresponds a cyclic $\PiGamma_I$-module of dimension vector~$n_I$. By applying this to the family given by the tautological bundle $T_I$ on $\mathcal{M}_{\PiQ_I}(1,n_I)$, we see that the universal property of the fine moduli space $\QuotInI\big(\big[\SC^2/\Gamma\big]\big)$
induces the inverse of the mor\-phism~$\omega_{n_I}$.
\end{proof}

\section{Quiver varieties for the framed McKay quiver}\label{sec:quivers}

\subsection{Variation of GIT quotient for quiver varieties}
\label{subsec:quivervars}

Let $v\coloneqq(v_i)_{0 \leq i \leq r} \in \mathbb{N}^{r+1}$ be a dimension vector for the McKay quiver $Q_\Gamma$.
Recall that we identify $\mathbb{Z}^{Q_0}\cong \mathbb{Z}\rho_{\infty}\oplus \Rep(\Gamma)$, so we obtain the
dimension vector \[(1,v) \coloneqq\rho_{\infty}+\sum_{i=0}^r v_i\rho_i \in \mathbb{N}^{Q_0}
\] for the framed McKay quiver $Q$. Write
\[ \Theta_{v}=\big\{\theta\colon \mathbb{Z}^{Q_0} \to \mathbb{Q}\; : \; \theta\big((1,v)\big)=0 \big\} \]
for the space of stability conditions for the
the framed quiver. Following Nakajima~\cite{nakajima1994instantons, Nakajima98}, for every $\theta \in \Theta_{v}$, the quiver variety $\mathfrak{M}_{\theta}(1,v)$ is the coarse moduli space of S-equivalence classes of $\theta$-semistable $\Pi$-modules of dimension vector $(1,v)$. The GIT construction of $\mathfrak{M}_{\theta}(1,v)$ is summarised for example in \cite[Section 2]{craw2019punctual}; see also \cite[Section 2]{nakajima2009quiver}. Note in particular that we give $\mathfrak{M}_{\theta}(1,v)$ the reduced scheme structure.

The set of stability conditions $\Theta_v$ admits a preorder $\geq$, and this determines a wall-and-chamber structure, where
$\theta,\theta' \in \Theta_v$ lie in the relative interior of the same cone if and only if both $\theta \geq \theta'$ and $\theta' \geq\theta$ hold in this preorder, in which case $\mathfrak{M}_{\theta}(1,v)\cong\mathfrak{M}_{\theta'}(1,v)$. The interiors of the top-dimensional cones in $\Theta_v$ are chambers, while the codimension-one faces of the closure of each chamber are walls. We say that $\theta \in \Theta_v$ is generic with respect to $v$, if it lies in some GIT chamber. The vector $(1,v)$ is indivisible, so again, \cite[Proposition~5.3]{king1994moduli} implies that for generic $\theta \in \Theta_v$ the quiver variety $\mathfrak{M}_{\theta}(1,v)$ is the fine moduli space of isomorphism classes of $\theta$-stable $\Pi$-modules of dimension vector $(1,v)$. In this case, the universal family on $\mathfrak{M}_{\theta}(1,v)$ is a tautological locally-free sheaf $\mathcal{R} \coloneqq \oplus_{i \in Q_0}\mathcal{R}_i$ together with a $\SC$-algebra homomorphism $\varphi \colon \Pi \to \mathrm{End}(\mathcal{R})$, where $R_\infty$ is the trivial bundle on $\mathfrak{M}_{\theta}(1,v)$ and where $\mathrm{rank}(\mathcal{R}_i) = v_j$ for $i \geq 0$.

\begin{Lemma}\label{lem:quiver}\quad
\begin{enumerate}\itemsep=0pt
\item[$1.$] For all $\theta \in \Theta_v$, the scheme $\mathfrak{M}_{\theta}(1,v)$ is irreducible and normal, with symplectic singularities.
\item[$2.$] Let $ \theta \geq \theta'$. The morphism $\pi\colon \mathfrak{M}_{\theta}(1,v) \to \mathfrak{M}_{\theta'}(1,v)$ obtained by variation of GIT quotient is a projective morphism of varieties over the affine quotient $\mathfrak{M}_{0}(1,v)$.
\end{enumerate}
\end{Lemma}

\begin{proof}
Part (1) is \cite[Theorem 1.2, Proposition 3.21]{bellamy2016symplectic} while part (2) is well-known.
\end{proof}

\subsection{Wall-and-chamber structure}
\label{subsec:wallandcham}

Define $\delta\coloneqq \sum_{0\leq i\leq r} \dim(\rho_i)\rho_i\in \Rep(\Gamma)$. For the dimension vector $v\coloneqq n\delta$, the GIT wall-and-chamber structure on $\Theta_{n\delta}$ is computed explicitly in \cite[Theorem 4.6]{bellamy2020birational}. More generally, for arbitrary $v \in \mathbb{N}^{r+1}$,
it is possible to compute the wall-and-chamber structure on $\Theta_v$ by applying recent work of
Dehority~\cite[Proposition 4.8]{dehority2020affinizations}. We do not need the wall-and-chamber structure here, but we do use the following distinguished region of the space of stability conditions:
\[ C^{+}_v\coloneqq \{\theta \in \Theta_v \mid \theta(\rho_i) > 0 \textrm{ for } 0 \leq i \leq r \}. \]
This open cone need not be a GIT chamber, though it is always contained in one (and it is a chamber in the special case $v=n\delta$, see \cite[Example~2.1]{bellamy2020birational} or \cite[Section~2.8]{nakajima2009quiver}).
The next statement, which appeared originally in \cite[Section~2]{nakajima2002geometric}, illustrates the importance of $C^{+}_v$.
\begin{Proposition}
\label{prop:hilbiso}
For $v \in \mathbb{N}^{r+1}$ and $\theta \in C^{+}_v$, there is an isomorphism
\[ \Hilb^{v}\big(\big[\SC^2/\Gamma\big]\big) \cong \mathfrak{M}_{\theta}(1,v). \]
\end{Proposition}
\begin{proof}
In the special case $v=n\delta$,
the statement is proved in \cite[Theorem 4.6]{kuznetsov2007quiver}. However, the argument given there applies for
an arbitrary dimension vector $v$; note in particular that the auxiliary \cite[Corollary 3.5]{kuznetsov2007quiver}
does not need any assumption on the dimension vector (and we have that $\Hilb^{v}\big(\big[\SC^2/\Gamma\big]\big)$ is non-empty if and only if $\mathfrak{M}_{\theta}(1,v)$ is non-empty).
\end{proof}

Every face of the closed cone $\overline{C^{+}_v}$ is of the form
\[ \sigma_I\coloneqq\big\{\theta \in \overline{C^{+}_v} \mid \theta(\rho_i) > 0 \textrm{ if and only if } i \in I, \textrm{ and } \theta(\rho_i) = 0 \textrm{ otherwise} \big\} \]
for some non-empty subset $I \subseteq\{0,\dots,r\}$. When $v=n\delta$, these faces are contained in walls of the wall-and-chamber structure on $\Theta_{n\delta}$, though this need not be the case for arbitrary $v \in \mathbb{N}^{r+1}$.
In any case, the parameter
\begin{equation} \label{eqn:theta_I}
\theta_I (\rho_j) \coloneqq
\begin{cases}
\displaystyle -\sum_{i \in I} v_i & \textrm{for }j=\infty, \\
1 & \textrm{if } j \in I, \\
0 & \textrm{if } j \in \{0,1,\dots,r\} \setminus I
\end{cases}
\end{equation}
lies in the relative interior of the face $\sigma_I$.

\subsection{Resolution of singularities}
Let $v \in \mathbb{N}^{r+1}$ be a dimension vector and $I \subseteq \{0, \dots , r\}$. As mentioned above, the quiver variety $\mathfrak{M}_{\theta_I}(1,v)$ is singular in general for $\theta_I$ as defined in \eqref{eqn:theta_I} above. Moreover, for $\theta\in C_v^+$,
the morphism
\[ \mathfrak{M}_{\theta}(1,v) \to \mathfrak{M}_{\theta_I}(1,v)\]
obtained by variation of GIT need not be birational.
It is nevertheless possible to find a resolution of $\mathfrak{M}_{\theta_I}(1,v)$ using a
quiver variety if one modifies the dimension vector~$v$:

\begin{Proposition}\label{prop:surjvar}
Let $v \in \mathbb{N}^{r+1}$ and let $ I \subseteq \{0, \dots , r\}$ be non-empty. Assume further that
$\mathfrak{M}_{\theta_I}(1,v)$ is non-empty.
There exists a dimension vector $\widetilde{v}\in \mathbb{N}^{r+1}$ satisfying $\widetilde{v}_i \leq v_i$ for $0 \leq i \leq r$ and $\widetilde{v}_i = v_i$ for $i \in I$,
such that there is a projective resolution of singularities
\[
\pi_{I}\colon \ \mathfrak{M}_{\widetilde{\theta}}(1,\widetilde{v})\to \mathfrak{M}_{\theta_I}(1,v),
\]
where the stability condition $\widetilde{\theta}$ satisfies $\widetilde\theta\in C_{\widetilde v}^+$.
\end{Proposition}
\begin{proof}
 Following Nakajima~\cite[Section 6]{nakajima1994instantons} (cf.\ \cite[Section 3.5]{bellamy2020birational}),
 there is a finite stratification
\begin{equation*}
\mathfrak{M}_{\theta_I}(1,v) = \coprod_{\gamma} \mathfrak{M}_{\theta_I}(1,v)_{\gamma} \end{equation*}
by smooth locally-closed subvarieties, where the union is over conjugacy classes of subgroups of the reductive group applied during the GIT construction of $\mathfrak{M}_{\theta_I}(1,v)$, and where the stratum $\mathfrak{M}_{\theta_I}(1,v)_{\gamma}$ consists of $S$-equivalence classes of semistable modules with polystable representative having stabiliser in the conjugacy class $\gamma$. Since $\mathfrak{M}_{\theta_I}(1,v)$ is irreducible by Lemma~\ref{lem:quiver}, there is a unique dense stratum and we fix a conjugacy class $\gamma$
 such that the dense open stratum in~$\mathfrak{M}_{\theta_I}(1,v)$ is $\mathfrak{M}_{\theta_I}(1,v)_\gamma$. Apply \cite[Proposition 2.25]{nakajima2009quiver} to obtain a dimension vector $\widetilde{v}\in \mathbb{N}^{r+1}$ as in the statement of the lemma
 together with an isomorphism
\begin{equation} \label{eqn:gammaIsom}
\mathfrak{M}_{\theta_I}(1,v)_{\gamma} \cong \mathfrak{M}^s_{\theta_I}(1,\widetilde{v}) \times {Y},
\end{equation}
 where $\mathfrak{M}^s_{\theta_I}(1,\widetilde{v})$ is the fine moduli space of $\theta_I$-stable $\Pi$-modules of dimension vector $(1,\widetilde{v})$, and ${Y} \subset \mathfrak{M}_{0}(0,v-\widetilde{v})$ is a locally-closed subvariety.
 Note that $\mathfrak{M}_{\theta_I}(1,v)_{\gamma}$ is non-empty by assumption, hence so are $\mathfrak{M}^s_{\theta_I}(1,\widetilde{v})$ and ${Y}$.

 We claim that ${Y}$ is a closed point. Indeed, $\mathfrak{M}_0(0, v-\widetilde{v})$ parametrises
 0-polystable representations, i.e., direct sums of simple representations. The framing is 0-dimensional,
and since~$I$ is non-empty, there is at least one $i$ such that $v_i = \widetilde{v}_i$. It follows that the representations parameterised by ${Y}$
are supported on a doubled quiver obtained from an affine ADE diagram with at least one vertex removed. Removing a vertex from an affine ADE diagram leaves a graph in which every connected component is a subgraph of an finite type ADE diagram. Therefore $\mathfrak{M}_0(0, v-\widetilde{v})$ parametrises direct sums of simple representations of preprojective algebras of doubled quivers associated to finite type ADE diagrams. A simple representation of any such quiver is one-dimensional \cite[Lemma~2.2]{savage2011quiver}, so the only polystable representation of dimension vector $(0, v-\widetilde{v})$ is
necessarily of the form
\begin{equation}\label{eqn:directsumsimples}
\bigoplus_{0\leq i\leq r}S_i^{\oplus(v_i-\tilde{v}_i)},
\end{equation}
where each $S_i$ is a vertex simple $\Pi$-module. Therefore $\mathfrak{M}_0(0, v-\widetilde{v})$ is a closed point, and hence so too is ${Y}$ because ${Y}\neq \varnothing$. Thus $\mathfrak{M}_{\theta_I}(1,v)_{\gamma} \cong \mathfrak{M}^s_{\theta_I}(1,\widetilde{v})$.

For
 $\widetilde\theta\in C_{\widetilde v}^+$, we claim that the desired morphism $\pi_I\colon \mathfrak{M}_{\widetilde{\theta}}(1,\widetilde{v})\to \mathfrak{M}_{\theta_I}(1,v)$ is induced by sending each $\tilde{\theta}$-stable $\Pi$-module $V$ of dimension vector $(1,\tilde{v})$ to
 \[
 \pi_I(V):=V\oplus\bigoplus_{0\leq i\leq r}S_i^{\oplus(v_i-\tilde{v}_i)}.
 \]
 Indeed, if $\tilde{\mu}$ and $\mu$ denote the moment maps for the actions of $G_{\tilde{v}}:=\prod_{0\leq i\leq r}\GL(\tilde{v}_i)$ and $G_{v}:=\prod_{0\leq i\leq r}\GL(v_i)$ on the representation spaces of $\Pi$-modules of dimension vectors $(1,\tilde{v})$ and $(1,v)$ respectively, then the above assignment induces an inclusion $\tilde{\mu}^{-1}(0)\hookrightarrow \mu^{-1}(0)$ that is equivariant with respect to the actions of $G_{\tilde{v}}$ and $G_v$ on $\tilde{\mu}^{-1}(0)$ and $\mu^{-1}(0)$ respectively. Any submodule of~$\pi_I(V)$ is $W\oplus W'$ for submodules $W\subseteq V$ and $W'\subseteq \bigoplus_{0\leq i\leq r}S_i^{\oplus(v_i-\tilde{v}_i)}$, and we have $\theta_I(W\oplus W') = \theta_I(W)\geq 0$. Thus, the image of the $\tilde{\theta}$-stable locus of $\tilde{\mu}^{-1}(0)$ lies in the $\theta_I$-semistable locus of~$\mu^{-1}(0)$, and this inclusion induces the morphism $\pi_I$ as claimed.

 It remains to establish the properties of $\pi_I$. Adapting the proof from \cite[Lemma~2.4]{bellamy2016symplectic} shows that~$\pi_I$ is a projective morphism, while \cite[Theorem~1.15]{bellamy2016symplectic} gives that $ \mathfrak{M}_{\widetilde{\theta}}(1,\widetilde{v})$ is non-singular as~$\widetilde{\theta}$ is generic. Our explicit description of~$\pi_I$ shows that it factors via the morphism
\begin{equation*}%\label{eqn:VGITthetatilde}
\mathfrak{M}_{\widetilde{\theta}}(1,\widetilde{v}) \to \mathfrak{M}_{\theta_I}(1,\widetilde{v})
\end{equation*}
obtained by variation of GIT quotient. This morphism is birational by \cite[Lemma~2.12]{nakajima2009quiver}. To see that the induced morphism $\mathfrak{M}_{\theta_I}(1,\widetilde{v}) \to \mathfrak{M}_{\theta_I}(1,v)$ is also birational, notice that its restriction to the (open dense) stable locus $\mathfrak{M}^s_{\theta_I}(1,\widetilde{v})$ simply adds the summand~\eqref{eqn:directsumsimples} to each representative $\Pi$-module, so it agrees with the isomorphism from \eqref{eqn:gammaIsom} above. Thus, the image of $\mathfrak{M}^s_{\theta_I}(1,\widetilde{v})$ under this induced morphism is $\mathfrak{M}_{\theta_I}(1,v)_\gamma$ which is dense in $\mathfrak{M}_{\theta_I}(1,v)$ by our choice of $\gamma$. This completes the proof.
 \end{proof}

Recall that $\mathfrak{M}_{\theta_I}(1,v)$ is a coarse moduli space of $\theta_I$-polystable modules up to
isomorphism \cite[Proposition 2.9(2)]{nakajima2009quiver}, and that every $\theta_I$-polystable module has a unique summand that has dimension 1 at the vertex $\infty$.
\begin{Lemma}
\label{lem:vwtildev}
Let $M$ be a $\theta_I$-polystable representative for an arbitrary class $[M] \in \mathfrak{M}_{\theta_I}(1,v)$, and let $M_{\infty}$ be the unique summand that has dimension 1 at the vertex $\infty$.
Then
\[ \dim_i\,M_{\infty} \leq \widetilde{v}_i\]
for $0 \leq i \leq r$ where $\widetilde{v}$ is as in Proposition~{\rm \ref{prop:surjvar}}.
\end{Lemma}
\begin{proof}
{The morphism $\pi_I$ from Proposition~\ref{prop:surjvar} is surjective, because quiver varieties are irreducible and $\pi_I$ is both proper and birational. Therefore $[M]=[\pi_I(V)]$ for some $\tilde{\theta}$-stable $\Pi$-module $V$ of dimension vector $(1,\tilde{v})$. Each vertex simple $S_i$ arising as a summand of $\pi_I(V)$} has dimension 0 at the vertex $\infty$, {so }we must have $M_{\infty} \subseteq {V}$. This implies the lemma.
\end{proof}

Suppose now that $\mathfrak{M}_{\theta_I}(1,v)$ is non-empty, and apply Proposition~\ref{prop:surjvar} to obtain $\tilde{v}\leq v$ and a resolution of singularities $\pi_{I}\colon \mathfrak{M}_{\widetilde{\theta}}(1,\widetilde{v})\to \mathfrak{M}_{\theta_I}(1,v)$ for $\theta\in C^{+}_{\tilde{v}}$. For any $\theta'\in \Theta_{{\tilde{v}}}$, there is a~line bundle $L_{C^{+}_{{\tilde{v}}}}(\theta')\coloneqq \bigotimes_{0\leq i\leq r} \det(\mathcal{R}_i)^{\theta'_i}$ on $\mathfrak{M}_{\tilde{\theta}}(1,{\tilde{v}})$. The line bundle $L_I\coloneqq L_{C^{+}_{\tilde{v}}}(\theta_I)$ plays an important role in what follows.
\begin{Lemma}\label{lem:factorimm}
Let $\theta\in C^{+}_{\tilde{v}}$, Then
 \begin{enumerate}\itemsep=0pt
 \item[$1)$] for each $\theta' \in \overline{C^{+}_{\tilde{v}}}$ the line bundle $L_{C^{+}_{\tilde{v}}}(\theta')$ on $\mathfrak{M}_{\tilde{\theta}}(1,{\tilde{v}})$ is globally generated;
 \item[$2)$] for any $I\subseteq \{0,\dots, r\}$, after multiplying $\theta_I$ by a positive integer if necessary, the morphism to the linear series of $L_I$ decomposes as the composition of the
 morphism $\pi_I$ from Proposition~{\rm \ref{prop:surjvar}} and a closed immersion:
 \begin{equation*}
 \begin{tikzcd}
 \mathfrak{M}_{\tilde{\theta}}(1,\tilde{v}) \ar[d,swap,"{\pi_I}"]\ar[r,"\varphi_{\vert L_I\vert}"] & \vert L_I\vert. \\ \mathfrak{M}_{\theta_I}(1,v)\ar[ur,hook] &
 \end{tikzcd}
 \end{equation*}
 \end{enumerate}
 \end{Lemma}
 \begin{proof}
 Since $\theta\in C^{+}_{\tilde{v}}$, the tautological bundles $\mathcal{R}_i$ on the quiver variety $\mathfrak{M}_{\tilde{\theta}}(1,{\tilde{v}})$ are globally generated for $i\in {Q_0}$ by \cite[Corollary~2.4]{craw2018multigraded}. Therefore, $L_{C^{+}_{\tilde{v}}}(\theta')$ is globally generated as $\theta_i'\geq 0$ for all $0\leq i\leq r$. In particular, since $\theta_I\in \overline{C^{+}_{\tilde{v}}}$, the rational map $\varphi_{\vert L_I\vert}$ is defined everywhere.

 For (2), after taking a positive multiple of $\theta_I$ if necessary, we may assume that the polarising ample
 line bundle $L^\prime_I$ on $\mathfrak{M}_{\theta_I}(1,v)$ is very ample. {Our explicit construction of $\pi_I$ from the proof of Proposition~\ref{prop:surjvar} implies that the $G_v$-equivariant line bundle on $\mu^{-1}(0)$ determined by the character associated to $\theta_I$ restricts under the inclusion $\tilde{\mu}^{-1}(0)\hookrightarrow \mu^{-1}(0)$ to the $G_{\tilde{v}}$-equivariant line bundle on $\tilde{\mu}^{-1}(0)$ determined by the character associated to $\theta_I$. Thus, after descent,
 we obtain $\pi_I^*(L^\prime_I) = L_I$ as required.}

 If $\varphi_{\vert L^\prime_I\vert}\colon \mathfrak{M}_{\theta_I}(1,v)\to \vert L^\prime_I\vert$ is the closed immersion, then
 \begin{equation*}
 (\varphi_{\vert L^\prime_I\vert}\circ \pi_I)^*(\mathcal{O}_{\vert L^\prime_I\vert}(1)) = \pi_I^*(L^\prime_I) = L_I = \varphi^*_{\vert L_I \vert}\big(\mathcal{O}_{\vert L_I\vert}(1)\big)
 \end{equation*}
 on $\mathfrak{M}_{\tilde{\theta}}(1,{\tilde{v}})$. The morphism to a complete linear series of a specific line bundle is unique up to an automorphism of the linear series \cite[Chapter~II, Theorem~7.1]{hartshorne1977algebraic}. Thus, after a change of basis on $H^0(L^\prime_I)$ if necessary, we have $\varphi_{\vert L_I \vert} = \varphi_{\vert L^\prime_I\vert}\circ \pi_I$ as required.
 \end{proof}

\section{Quiver varieties and moduli spaces for cornered algebras}\label{sec:quivfine}
\subsection{The key commutative diagram}
Once and for all, fix a non-empty $I\subseteq \{0,1,\dots,r\}$
and a dimension vector $n_I=(n_i)_{i \in I} \in \mathbb{N}^{I}$.
 \begin{Proposition}
 \label{prop:key}
 Let $v=(v_j)_{0 \leq j \leq r} \in \mathbb{N}^{r+1}$ be any vector satisfying $v_i=n_i$ for all $i \in I$. Suppose that $\mathfrak{M}_{\theta_I}(1,v)$ is non-empty, and define $\widetilde{v}$ as in Proposition~{\rm \ref{prop:surjvar}}. Then $\widetilde{v}$ satisfies $\widetilde{v}_i=n_i$ for $i\in I$, and there is a commutative diagram
\begin{equation*}
\begin{tikzcd}
 & \mathfrak{M}_{\widetilde{\theta}}(1, \widetilde{v}) \arrow[dl, "\pi_I", swap ] \arrow[dr, "\tau_I" ]& \\
 \mathfrak{M}_{\theta_I}(1,v) \arrow[rr, "\iota_I" ] & & \mathcal{M}_{\PiQ_I} (1,n_I)
\end{tikzcd}
\end{equation*}
 of schemes over $\C$ for any $\widetilde{\theta}\in C^+_{\widetilde{v}}$,
where $\pi_I$ is surjective and where $\iota_I$ is a closed immersion. In particular, the fine moduli space $\mathcal{M}_{\PiQ_I}(1,n_I)$ is non-empty.
\end{Proposition}

 \begin{proof}
The fact that $\widetilde{v}_i=n_i$ for all $i\in I$ is immediate from Proposition~\ref{prop:surjvar}. We now construct a commutative diagram
\begin{equation*}
 \begin{tikzcd}
 & \mathfrak{M}_{\widetilde{\theta}}(1,\widetilde{v}) \ar[dl,swap,"{\pi_I}"] \ar[dr,"{\tau_I}"]\ar[dd,"\varphi_{\vert L_I\vert}"] & \\ \mathfrak{M}_{\theta_I}(1,v)\ar[rd,hook,"\sigma_I"] & & \mathcal{M}_{\PiQ_I}(1,n_I)\ar[dl,hook,"\varphi_I"] \\
 & \vert L_I\vert &
 \end{tikzcd}
 \end{equation*}
 of schemes over $\C$, where both morphisms drawn diagonally in the bottom half of the diagram are closed immersions.
 Commutativity of the triangle on the left is given by Lemma~\ref{lem:factorimm}, while surjectivity of $\pi_I$
 is established in the proof of Lemma~\ref{lem:vwtildev}.

The isomorphism $\Hilb^{\widetilde{v}}\big(\big[\SC^2/\Gamma\big]\big) \cong \mathfrak{M}_{\theta}(1,\widetilde{v})$ from Proposition~\ref{prop:hilbiso} allows us to apply the proof of \cite[Lemma~3.1]{craw2019punctual}, using $\widetilde{v}$ in place of $n\delta$ throughout, to conclude that
$\mathfrak{M}_{\tilde{\theta}}(1,\widetilde{v})$ may be regarded as a fine moduli space of $A$-modules. Applying the proof of \cite[Lemma 4.1]{craw2019punctual}, with $\widetilde{v}_i$ replacing $n \mathrm{dim}\,\rho_i$ for all $i\in I$, gives the universal morphism $\tau_I$.

For commutativity of the right-hand triangle, recall that the polarising line bundle $\mathcal{L}_I$ on $\mathcal{M}_{\PiQ_I}(1,n_I)$ is very ample. As pullback commutes with tensor operations on the $T_i$, the isomorphisms $\tau_I^*(T_i) \cong \mathcal{R}_i$ for $i\in \{\infty\}\cup I$ from
{\emph{ibid.}}\ imply that $L_I = \tau_I^*(\mathcal{L}_I)$. If $\varphi_{\vert \mathcal{L}_I\vert}\colon \mathcal{M}_{\PiQ_I}(1,n_I)\to \vert \mathcal{L}_I\vert$ is the closed immersion for $\mathcal{L}_I$, then
 \begin{equation*}
 (\varphi_{\vert \mathcal{L}_I\vert}\circ \tau_I)^*(\mathcal{O}_{\vert \mathcal{L}_I\vert}(1)) = \tau_I^*(\mathcal{L}_I) = L_I = \varphi^*_{\vert L_I \vert}\big(\mathcal{O}_{\vert L_I\vert}(1)\big)
 \end{equation*}
 on $\mathfrak{M}_\theta(1,v)$. Now, just as in the proof of Lemma~\ref{lem:factorimm}, we deduce $\varphi_{\vert L_I \vert} = \varphi_{\vert \mathcal{L}_I\vert}\circ \tau_I$, so the right-hand triangle commutes and hence so does the diagram.

 It remains to construct the closed immersion $\iota_I$. Commutativity of the diagram, combined with surjectivity of $\pi_I$, shows that $\mathfrak{M}_{\theta_I}(1,v)$ is isomorphic to the closed subscheme $\mathrm{Im}(\sigma_I)= \mathrm{Im}(\sigma_I\circ \pi_I) = \mathrm{Im}(\varphi_{I}\circ \tau_I)$ of $\vert L_I\vert$. The closed immersion $\varphi_I$ induces an isomorphism $\lambda_I\colon \mathrm{Im}(\varphi_I) \to \mathcal{M}_{A_I}(1,n_I)$ and hence we obtain a closed immersion
 \[
 \iota_I\coloneqq \lambda_I\circ \sigma_I\colon \mathfrak{M}_{\theta_I}(1,v)\to \mathcal{M}_{\PiQ_I}(1,n_I).
 \]
 Since $\mathfrak{M}_{\theta_I}(1,v)$ is non-empty by assumption, it follows that $\mathcal{M}_{\PiQ_I}(1,n_I)$ is non-empty.
 \end{proof}

\subsection{Selecting a suitable dimension vector}
We don't yet know whether the morphism $\tau_I$ from the key commutative diagram from Proposition~\ref{prop:key} is surjective. We now introduce a collection of dimension vectors that will allow us to establish surjectivity.

For any vector $v\in \mathbb{N}^{r+1}$, it is convenient to write $v_\infty \coloneqq 1$, so that the dimension vector $(1,v)$ has components $v_i$ for all $i\in Q_0$.

 \begin{Lemma}
 \label{lem:Vnj}
 Let $v=(v_j)_{0 \leq j \leq r} \in \mathbb{N}^{r+1}$ satisfy $v_i=n_i$ for all $i \in I$. Let $ V $ be a $ \theta_I $-stable $A$-module
 of dimension vector $(1,v)$. Then $2v_k \leq \sum_{\{e \in Q_1\mid \tail(e)=k\}} v_{\head(e)}$ for all $k \notin \{\infty\}\cup I$.
 \end{Lemma}
 \begin{proof}
 Fix $k \notin \{\infty\}\cup I$ and define
 \[
 	V_\oplus\coloneqq\bigoplus_{\substack{e\in {Q_1},\\ \; \tail(e)=k }} V_{\head(e)}.
 	\]
 	The maps in $V$ determined by arrows with head and tail at vertex $k$ combine to define maps $f\colon V_k\to V_\oplus$ and $g\colon V_\oplus\to V_k$ satisfying $g\circ f = 0$.
 	The same proof as in \cite[Proposition A.2]{craw2019punctual} implies that the complex
 	\[
 	0 \longrightarrow V_k\stackrel{f}{\longrightarrow} V_\oplus \stackrel{g}{\longrightarrow} V_k \longrightarrow 0
 	\]
 	has nonzero homology only at $V_\oplus$, so $\dim V_\oplus \geq 2\dim V_k$. This proves the claim.
 \end{proof}

For our chosen vector $n_I=(n_i)_{i \in I} \in \mathbb{N}^{I}$, define
 \[
\mathcal{V}(n_I) \coloneqq\left\{(1,v)\in \mathbb{N}^{r+2} \mid v_i = n_i \text{ for } i\in I \text{ and }2v_k \geq \sum_{\{e \in Q_1\mid \tail(e)=k\}} v_{\head(e)} \text{ for }k\not\in \{\infty\}\cup I \right\}
\]
 (recall that $v_\infty=1$). Before we prove that this set is non-empty, recall from McKay~\cite{McKay80} that for each vertex $0\leq k\leq r$ in the McKay quiver, we have
\begin{equation}
 \label{eqn:McKayequality}
2\dim(\rho_k) = \sum_{\{e\in Q_1^* \mid \tail(e)=k\}} \dim(\rho_{\head(e)}).
\end{equation}
Recall also that $Q_1=\{b^*\}\cup Q_1^*$, so the indexing set in the sum from \eqref{eqn:McKayequality} differs from the set $\{e \in Q_1\mid \tail(e)=0\}$ only for $k=0$.

\begin{Example}
\label{ex:ndeltaVnI}
If there exists $n>0$ such that $n_i = n\dim \rho_i$ for all $i\in I$, then equation \eqref{eqn:McKayequality} shows that the vector $(1,v)$ with $v_k =n\dim(\rho_k)$ for all $0\leq k\leq r$ lies in $\mathcal{V}(n_I)$. This class of dimension vectors appears in \cite{craw2019punctual}.
\end{Example}

\begin{Lemma}
\label{lem:VnInon-empty}
For any $v\in \mathbb{N}^{r+1}$ satisfying $v_i=n_i$ for all $i\in I$, there exists $(1,\vprime)\in \mathcal{V}(n_I)$ such that $\vprime-v\geq 0$.
\end{Lemma}
\begin{proof}
Write $K\coloneqq \{0,1,\dots,r\}\setminus I$. Set $\vprime_\infty = 1$ and $\vprime_i=n_i$ for $i\in I$. Our goal is to define $\vprime_k\in \mathbb{N}$ for each $k\in K$ such that $\vprime_k\geq v_k$ and
the inequality
 \begin{equation}
 \label{eqn:VnIinequality}
 2\vprime_k \geq \sum_{\{e \in Q_1\mid \tail(e)=k\}} \vprime_{\head(e)}
 \end{equation}
 holds for all $k\in K$.

 For this, define a subset $K^\prime\subseteq K$ as follows. If $0\in I$, then set $K^\prime=\varnothing$. Otherwise, we have $0\in K$. In this case, choose any shortest path $\gamma$ in the McKay graph starting at vertex~$0$ and ending at some vertex $i\in I$. The set $K^\prime$ of vertices through which $\gamma$ passes satisfies $K^\prime\subseteq K\setminus \{0\}$. Now,
 fix $N\gg 0$ and for each $k\in K^\prime$, let $d_k$ denote the distance in the graph from~$0$ to~$k$. For $k\in K$, define{\samepage
 \[
 \vprime_k= \begin{cases} N\dim(\rho_k) - d_k & \text{if }k\in K^\prime, \\
 N\dim(\rho_k) & \text{otherwise.}
 \end{cases}
 \]
 Since $N\gg 0$, we have $\vprime_k\geq v_k$ for all $k\in K$. Also, for all $k\in K$ and $i\in I$, we have $\vprime_k\gg \vprime_i$.}

 It remains to establish the inequality \eqref{eqn:VnIinequality} for all $k\in K$. To do this, consider three cases, depending on whether $k\in K'$ and whether $k=0$.

 First, suppose $k\in K^\prime$. If there is a vertex $j\in I$ that lies adjacent to $k$, then \eqref{eqn:VnIinequality} follows by combining \eqref{eqn:McKayequality} with the inequality $\vprime_k\gg \vprime_j$. Otherwise, precisely two vertices adjacent to $k$ lie in $\{0\}\cup K^\prime$, while every other vertex adjacent to $k$ lies in $K$. Then
\begin{align*}
2\vprime_k
= 2(N\dim(\rho_k) - d_k) & = -(d_{k}-1)-(d_{k}+1) +\sum_{\{e\in Q_1^*\mid \tail(e)=k\}}N \dim(\rho_{\head(e)})\\
& = \sum_{\{e\in Q_1^*\mid \tail(e)=k\}} \vprime_{\head(e)},
\end{align*}
 so \eqref{eqn:VnIinequality} holds because $k\neq 0$.

 Next, suppose $k\in K\setminus K'$ with $k\neq 0$. Then each vertex $j$ adjacent to $k$ lies in $\{0,1,\dots,r\}$ and satisfies $N\dim(\rho_j)\geq v_j'$. Combine these inequalities with \eqref{eqn:McKayequality} to obtain
 \[
 2\vprime_k = 2N\dim(\rho_k) = \sum_{\{e\in Q_1^* \mid \tail(e)=k\}} N\dim(\rho_{\head(e)})
\geq \sum_{\{e\in Q_1^* \mid \tail(e)=k\}} \vprime_{\head(e)}.
\]
This establishes the inequality \eqref{eqn:VnIinequality} since $k\neq 0$.

 Finally, the only remaining case is when $k=0\in K$. If there is a vertex $j\in I$ that lies adjacent to $0$, then \eqref{eqn:VnIinequality} follows by combining \eqref{eqn:McKayequality} with the inequality $\vprime_0\gg \vprime_j$ and $\vprime_\infty = 1$. Otherwise, we have $0\in K$ and one vertex of $K^\prime$ lies adjacent to $0$, giving
 \[
 2\vprime_0 = 2N = \vprime_
\infty + (-1) + \sum_{\{e\in Q_1^* \mid \tail(e)=0\}} N \dim(\rho_{\head(e)}) = \vprime_\infty + \sum_{\{e\in Q_1^*\mid \tail(e)=0\}} \vprime_{\head(e)}.
 \]
 The quiver $Q_1$ has a unique arrow with tail at $0$ and head at $\infty$, so \eqref{eqn:VnIinequality} holds for $k=0$ if $0\in K$.
 This completes the proof.
 \end{proof}

 \begin{Corollary}
 \label{cor:non-empty}
 Let $v\in \mathbb{N}^{r+1}$ satisfy $v_i=n_i$ for all $i\in I$.
 If $\mathfrak{M}_{\theta_I}(1,v)$ is non-empty, then there exists $(1,\vprime)\in \mathcal{V}(n_I)$ such that $\vprime-v \geq 0$ and $\mathfrak{M}_{\theta_I}(1,\vprime)$ is non-empty.
 \end{Corollary}
 \begin{proof}
Lemma~\ref{lem:VnInon-empty} determines a vector $(1,\vprime)\in \mathcal{V}(n_I)$ such that
 $\vprime_k-v_k\geq 0$ for all $k\in K\coloneqq\{0,1,\dots,r\}\setminus I$. Take any $\theta_I$-semistable $A$-module $M$ of dimension vector $(1,v)$, and let $S_k\coloneqq\SC e_k$ denote the vertex simple $A$-module at vertex $k\in K$.
 The $\PiQ$-module
 \[
 \overline{M}\coloneqq M\oplus \bigoplus_{k\in K}
 S_k^{\oplus (\vprime_k-v_k)}
 \]
 is $\theta_I$-semistable of dimension vector $(1,\vprime)$ by construction.
 \end{proof}

\subsection{Quiver varieties as fine moduli spaces}
As before, $n_I=(n_i)\in \mathbb{N}^I$ is a dimension vector and $\eta_I\colon \mathbb{Z}\oplus \mathbb{Z}^I\to \mathbb{Q}$ is the stability condition for $A_I$-modules of dimension vector $(1,n_I)$ defined in \eqref{eqn:etaI}. Let $v=(v_j)_{0\leq j\leq r}\in \mathbb{N}^{r+1}$ be a~dimension vector satisfying $v_i=n_i$ for all $i\in I$, and let $\theta_I\in \Theta_v$ be the stability condition for $A$-modules defined in~\eqref{eqn:theta_I}.

We now use Corollary~\ref{cor:non-empty} to strengthen the statement of Proposition~\ref{prop:key}. Before stating the desired result, recall the idempotent $\overline{e}_I=e_\infty+\sum_{i\in I} e_i\in A$ and the algebra $A_I=\overline{e}_I A\overline{e}_I$. Consider the functors
 \[
\begin{tikzpicture}
\node (v1) at (0,0) {\scriptsize ${{\PiQ}\module}$};
\node (v2) at (4,0) {\scriptsize ${{\PiQ_I}\module}$};

\draw [->,bend right=-5] (v1) to node[above] {$\scriptstyle{j_I^*}$} (v2);
\draw [->,bend left=5] (v2) to node[below] {$\scriptstyle{j_{I!}}$} (v1);
\end{tikzpicture}
\]
 defined by $j_I^*(-)\coloneqq \overline{e}_I \PiQ \otimes_{\PiQ} (-)$ and $j_{I!}(-)\coloneqq \PiQ \overline{e}_I\otimes_{\PiQ_I}(-)$. These are two of the six functors in a recollement of the module category ${\PiQ}\module$ \cite[Section 3]{craw2018multigraded}. In particular, $j_I^*$ is exact, $j_I^*j_{I!}$ is the identity functor, and for any $\PiQ_I$-module $N$, the $\PiQ$-module $j_{I!}(N)$ provides the maximal extension by $\PiQ/(\PiQ \overline{e}_I\PiQ)$-modules.

\begin{Lemma}\label{lem:j!semistable}
Let $N$ be an $\eta_I$-stable $\PiQ_I$-module of dimension vector $(1,n_I)$. Then $j_{I!}(N)$ is a~$\theta_I$-semistable $\PiQ$-module of finite dimension satisfying $\dim_i j_{I!}(N)=\dim_i N$ for $i \in \{\infty\} \cup I$.
\end{Lemma}
\begin{proof}The proofs of \cite[Lemma~4.4 and Lemma~4.5]{craw2019punctual} do not rely on the choice of dimension vector, so they apply equally well for the vector $(1,n_I)$.
\end{proof}

This gives the following partial converse to the Proposition~\ref{prop:key}.

 \begin{Corollary} Suppose that $\mathcal{M}_{A_I}(1,n_I)$ is non-empty. There exists $v\in \mathbb{N}^{r+1}$ satisfying $v_i=n_i$ for all $i \in I$ such that $\mathfrak{M}_{\theta_I}(1,v)$ is non-empty.
 \end{Corollary}
 \begin{proof}
 Let $N$ denote an $\eta_I$-stable $A_I$-module of dimension vector $(1,n_I)$. Lemma~\ref{lem:j!semistable} shows that $j_{I!}(N)$ is a $\theta_I$-semistable $A$-module of dimension vector $(1,v)$, where $v\in \mathbb{N}^{r+1}$ satisfies $v_i=n_i$ for all $i\in I$. In particular, $\mathfrak{M}_{\theta_I}(1,v)\neq \varnothing$.
 \end{proof}

 \begin{Lemma} \label{lem:existsemistable} Let $N$ be an $\eta_I$-stable $\PiQ_I$-module of dimension vector $(1,n_I)$ and let $(1,v) \in \mathcal{V}(n_I)$ be any vector such that $\mathfrak{M}_{\theta_I}(1,v)$ is non-empty. Choose
$\widetilde{v}$ as in Proposition~{\rm \ref{prop:surjvar}}. Then there exists a $\theta_I$-semistable $\PiQ$-module $M$ such that $j_I^*M \cong N$ with $\dim_{\infty} M =1$, $\dim_i M = n_i$ for all $i\in I$ and $\dim_k M \leq \widetilde{v}_k$ for all $k\in K\coloneqq\{0,1,\dots,r\}\setminus I$.
 \end{Lemma}

\begin{proof} Lemma~\ref{lem:j!semistable} shows that $j_{I!}(N)$ is $\theta_I$-semistable. If $\dim_i j_{I!}(N) \leq \widetilde{v}_i$ for $i\not\in \{\infty\}\cup I$, then
 $M\coloneqq j_{I!}(N)$ is the required $A$-module because $j_I^*j_{I!}$ is the identity. Otherwise, consider the $\theta_I$-polystable module $\bigoplus_{\lambda} M_\lambda$ that is $S$-equivalent to $j_{I!}(N)$. Let $M_{\lambda_\infty}$ denote the unique summand satisfying $\dim_\infty M_{\lambda_\infty} =1$. By Lemma~\ref{lem:j!semistable}, we have $\dim_i j_!(N) = n_i$, so $\dim_i M_{\lambda_\infty} \le n_i$.
 If $\dim_i M_{\lambda_\infty}<n_i$ for some $i\in I$, we would have $\theta_I(M_{\lambda_\infty})<0$, contradicting the
 $\theta_I$-stability of~$M_{\lambda_\infty}$. It follows that $\dim_i M_{\lambda_\infty} =n_i$ for all $i\in I$, and hence $\dim_i M_\lambda =0$ for all $\lambda\neq {\lambda_\infty}$ and all $i\in \{\infty\}\cup I$. For each index $\lambda$ and all $i\in \{\infty\}\cup I$, we have
\[
\dim_i j_I^*M_\lambda = \dim e_i \big(\overline{e}_I\PiQ\otimes_{\PiQ}(M_\lambda)\big) = \dim e_i \PiQ\otimes_{\PiQ} M_\lambda = \dim_i M_\lambda,
\]
 so $\dim_i j_I^*M_\lambda = 0$ for all
 $\lambda\neq {\lambda_\infty}$ and $i\in \{\infty\}\cup I$, and hence
 $j_I^*M_\lambda=0$ for $\lambda\neq {\lambda_\infty}$.

 We claim that $j_I^*M_{\lambda_\infty}$ is isomorphic to $N$. Indeed, the $\PiQ$-module $j_{I!}(N)$ is $\theta_I$-semistable by Lemma~\ref{lem:j!semistable}, and the $\theta_I$-stable $\PiQ$-modules $M_\lambda$ are
 the factors in the composition series of~$j_{I!}(N)$ in the category of $\theta_I$-semistable $\PiQ$-modules. It follows from exactness of $j_I^*$ that the $\PiQ_I$-modules~$j_I^*M_\lambda$ are the factors in the composition series of $j_I^*j_{I!}(N)\cong N$ in the category of $\eta_I$-semistable $\PiQ_I$-modules. But $j_I^*M_\lambda=0$ for $\lambda\neq {\lambda_\infty}$, so the only nonzero factor of the composition series is $j_I^*M_{\lambda_\infty}$. It follows that $j_I^*M_{\lambda_\infty}\cong N$, because the factor $j_I^*M_{\lambda_\infty}$ can only appear once in the composition series.

 We established above that
 $\dim_\infty M_{\lambda_\infty} =1$ and $\dim_i M_{\lambda_\infty} = n_i$ for all $i\in I$.
 Thus, to show that $M_{\lambda_\infty}$ is the required $A$-module, it remains to check that
 \begin{equation*}% \label{eqn:dimMlambdainfty}
 \dim_k M_{\lambda_\infty} \leq \widetilde{v}_k
 \end{equation*}
 for $k\in K=\{0,1,\dots,r\}\setminus I$.
 To simplify notation, write
 $(1,w)\coloneqq \dim M_{\lambda_\infty}$ for some $ w \in \SN^{r+1}$. We first establish the weaker inequality $ w \leq v$ (recall that $\widetilde{v} \leq v$ by Proposition~\ref{prop:surjvar}).
Combine the inequalities on the components of $v$ coming from the assumption $(1,v)\in \mathcal{V}(n_I)$, together with the inequalities on the components of $(1,w)$ obtained by applying Lemma~\ref{lem:Vnj} to the $\theta_I$-stable $\PiQ_I$-module $M_{\lambda_\infty}$, to obtain the inequality
 \begin{equation} \label{eqn:v-wk}
 2(v-w)_k \geq \sum_{\{e \in Q_1\mid \tail(e)=k\}} (v-w)_{\head(e)}
 \end{equation}
 for all $k \in K\coloneqq \{0,\dots,r\}\setminus I$.
 The vertex set $K$, together with all edges of the ADE Dynkin diagram joining two vertices from $K$, is
 a disjoint union of simply laced Dynkin diagrams. Let~$\Lambda$ be one of these subdiagrams, and let $C$ be its Cartan matrix. It follows from \eqref{eqn:v-wk} that $C(v-w)|_\Lambda\geq 0$. But the matrix $C^{-1}$ only has nonnegative entries (see, for example, \cite[equations~(1.157)--(1.158)]{Rosenfeld}), so $(v-w)|_\Lambda=C^{-1}C(v-w)|_\Lambda\ge 0$. Combining these inequalities for every such subdiagram $\Lambda$ with the equalities $ w_i=v_i$ for $i\in \{\infty\}\cup I$ gives $v\ge w$. This shows that $v_k-w_k\geq 0$ for all $k\in K$. Let $S_k\coloneqq\SC e_k$ denote the vertex simple $A$-module at vertex $k\in K$. The $\PiQ$-module
 \[
 \overline{M}\coloneqq M_{\lambda_\infty}\oplus \bigoplus_{k\in K}
 S_k^{\oplus (v_k-w_k)}
 \]
 is $\theta_I$-semistable of dimension vector $(1,v)$ by construction. As $M_{\lambda_\infty}$ is the unique summand of $\overline{M}$ with
dimension 1 at the vertex $\infty$, we have $w \leq \widetilde{v}$ by Lemma~\ref{lem:vwtildev}.
 \end{proof}

\begin{Theorem}\label{thm:mthetamorph1}
Let $v=(v_j)_{0 \leq j \leq r} \in \mathbb{N}^{r+1}$ be any vector satisfying $v_i=n_i$ for all $i \in I$. Suppose that $\mathfrak{M}_{\theta_I}(1,v)$ is non-empty. For the vector $(1,\vprime)\in \mathcal{V}(n_I)$ from Corollary~{\rm \ref{cor:non-empty}}, choose a vector~$\widetilde{\vprime}$ as in Proposition~{\rm \ref{prop:surjvar}}. Then
there is a commutative diagram
\begin{equation}\label{eqn:keyrevisited}
\begin{tikzcd}
 & \mathfrak{M}_{\widetilde{\theta}}\big(1, \widetilde{\vprime}\big) \arrow[dl, "\pi_I", swap ] \arrow[dr, "\tau_I" ]& \\
 \mathfrak{M}_{\theta_I}(1, \vprime) \arrow[rr, "\iota_I" ] & & \mathcal{M}_{\PiQ_I} (1,n_I)
\end{tikzcd}
\end{equation}
 of schemes over $\C$, where $\iota_I$ is an isomorphism of the underlying reduced schemes. Thus, $\mathcal{M}_{\PiQ_I}(1,n_I)$ is non-empty, irreducible, and its
underlying reduced scheme is normal and has symplectic singularities.
\end{Theorem}
 \begin{proof}
Note first that $\mathfrak{M}_{\theta_I}(1,\vprime)$ is non-empty by Corollary~\ref{cor:non-empty}, so we may indeed apply Proposition~\ref{prop:surjvar} to obtain the vector $\widetilde{\vprime}$. The diagram
 from Proposition~\ref{prop:key} now takes the form as shown in \eqref{eqn:keyrevisited}. To prove that $\iota_I$ is an isomorphism of the underlying reduced schemes, it suffices to show that $\iota_I$ is surjective on closed points. The diagram commutes, $\pi_I$ is surjective and $\iota_I$ is a~closed immersion, so it suffices to show that $\tau_I$ is surjective on closed points. Consider a closed point $[N]\in \mathcal{M}_{\PiQ_I}(1,n_I)$, where $N$ is an $\eta_I$-stable $\PiQ_I$-module of dimension vector $(1,n_I)$. Let $M$ be the $\theta_I$-semistable $\PiQ$-module from Lemma~\ref{lem:existsemistable}. Since
 \[
 m_k\coloneqq {\vprime_k}-\dim_k M \geq {\widetilde{\vprime_k}} - \dim_k M \geq 0
 \]
 for all $k\in K=\{0,1,\dots,r\}\setminus I$ by Proposition~\ref{prop:surjvar} and Lemma~\ref{lem:existsemistable}, we may once again define
 \[
 \overline{M}\coloneqq M\oplus \bigoplus_{k\in K} S_k^{\oplus m_k},
 \]
 where
 $S_k$ is the vertex simple $A$-module at vertex $k\in K=\{0,\dots,r\}\cup I$. The $\PiQ$-module $ \overline{M}$ is $\theta_I$-semistable of dimension vector $(1,\vprime)$ by construction, and it satisfies $j_I^*(\overline{M}) = j_I^*(M) = N$. Write $[\overline{M}]\in \mathfrak{M}_{\theta_I}(1,\vprime)$ for the corresponding closed point, and let $\widetilde{M}$ be any $\widetilde{\theta}$-stable $A$-module of dimension vector $\big(1, \widetilde{\vprime}\big)$ such that the closed point $[\widetilde{M}]\in \mathfrak{M}_{\widetilde{\theta}}\big(1, \widetilde{\vprime}\big)$ satisfies $\pi_I([\widetilde{M}])=[\overline{M}]\in \mathfrak{M}_{\theta_I}(1,\vprime)$. Then $j_I^*(\widetilde{M})=j_I^*(\overline{M})=N$, hence $\tau_I([\widetilde{M}])= [N]$.
 The final statement of the proposition follows from Lemma~\ref{lem:quiver} and Proposition~\ref{prop:key}.
 \end{proof}

\begin{Corollary}
\label{cor:MfMnonempty}
Let $I\subseteq \{0,\ldots, r\}$ be a non-empty subset and let $n_I\in {\mathbb N}^I$. Then
\[
\mathcal{M}_{\PiQ_I} (1,n_I)\neq \varnothing \iff \mathfrak{M}_{\theta_I}(1,v)\neq \varnothing\text{ for some }v\in \mathbb{N}^{r+1}\text{ satisfying }v_i=n_i\text{ for }i\in I.
\]
\end{Corollary}
\begin{proof}
If $\mathcal{M}_{\PiQ_I} (1,n_I)\neq \varnothing$, then Corollary~\ref{cor:non-empty} gives a vector $v\in \mathbb{N}^{r+1}$ satisfying $v_i=n_i$ for $i\in I$ such that $\mathfrak{M}_{\theta_I}(1,v)\neq \varnothing$. The converse is immediate from Proposition~\ref{prop:key}.
\end{proof}

Finally, we prove the results announced in Section~\ref{sec:intro}.
\begin{proof}[Proof of Theorem~\ref{thm:mainintro}]
The isomorphism in~(1), including the case when either space is empty, follows from Proposition~\ref{prop:omegaiso}. The isomorphism in~(2) is proved in Proposition~\ref{prop:Quotcommiso}. Statement~(3) follows by combining Corollary~\ref{cor:MfMnonempty} and Theorem~\ref{thm:mthetamorph1}. The final statement of Theorem~\ref{thm:mainintro} now
follows from \cite[Theorem~1.5]{bellamy2016symplectic}.
\end{proof}

\subsection*{Acknowledgements} The authors are grateful to Michel van den Bergh, Hiraku Nakajima, Yukinobu Toda, Michael Wemyss and the anonymous referees for questions, comments and suggestions. A.C.\ was supported by the Leverhulme Trust grant RPG-2021-149; S.G.~was supported by an Aker Scholarship; \'A.Gy.~and B.Sz.~were supported by
the EPSRC grant EP/R045038/1. \'A.Gy.~was also supported by the European Union's Horizon 2020
research and innovation programme under the Marie Sk{\l}odowska-Curie grant
agreement No.\ 891437.

\pdfbookmark[1]{References}{ref}
\LastPageEnding


\begin{thebibliography}{99}
\footnotesize\itemsep=0pt

\bibitem{Auslander62}
Auslander M., On the purity of the branch locus, \href{https://doi.org/10.2307/2372807}{\textit{Amer.~J. Math.}}
 \textbf{84} (1962), 116--125.

\bibitem{beentjes2018virtual}
Beentjes S.V., Ricolfi A.T., Virtual counts on {Q}uot schemes and the higher
 rank local {DT/PT} correspondence, \textit{Math. Res. Lett.}, {t}o appear,
 \href{https://arxiv.org/abs/1811.09859}{arXiv:1811.09859}.

\bibitem{bellamy2020birational}
Bellamy G., Craw A., Birational geometry of symplectic quotient singularities,
 \href{https://doi.org/10.1007/s00222-020-00972-9}{\textit{Invent. Math.}} \textbf{222} (2020), 399--468, \href{https://arxiv.org/abs/1811.09979}{arXiv:1811.09979}.

\bibitem{bellamy2016symplectic}
Bellamy G., Schedler T., Symplectic resolutions of quiver varieties,
 \href{https://doi.org/10.1007/s00029-021-00647-0}{\textit{Selecta Math.~(N.S.)}} \textbf{27} (2021), 36, 50~pages,
 \href{https://arxiv.org/abs/1602.00164}{arXiv:1602.00164}.

\bibitem{BourbakiLie}
Bourbaki N., Lie groups and {L}ie algebras. {C}hapters 4--6, \textit{Elements of
 Mathematics (Berlin)}, \href{https://doi.org/10.1007/978-3-540-89394-3}{Springer-Verlag}, Berlin, 2002.

\bibitem{Buchweitz}
Buchweitz R.-O., From {P}latonic solids to preprojective algebras via the
 {M}c{K}ay correspondence, in Oberwolfach Jahresbericht, Mathematisches
 Forschungsinstitut Oberwolfach, 2012, 18--28, available at
 \url{https://publications.mfo.de/handle/mfo/475}.

\bibitem{craw2019punctual}
Craw A., Gammelgaard S., Gyenge \'A., Szendr\H{o}i B., Punctual {H}ilbert schemes
 for {K}leinian singularities as quiver varieties, \href{https://doi.org/10.14231/AG-2021-021}{\textit{Algebr. Geom.}}
 \textbf{8} (2021), 680--704, \href{https://arxiv.org/abs/1910.13420}{arXiv:1910.13420}.

\bibitem{craw2018multigraded}
Craw A., Ito Y., Karmazyn J., Multigraded linear series and recollement,
 \href{https://doi.org/10.1007/s00209-017-1965-1}{\textit{Math.~Z.}} \textbf{289} (2018), 535--565, \href{https://arxiv.org/abs/1701.01679}{arXiv:1701.01679}.

\bibitem{crawley1998noncommutative}
Crawley-Boevey W., Holland M.P., Noncommutative deformations of {K}leinian
 singularities, \href{https://doi.org/10.1215/S0012-7094-98-09218-3}{\textit{Duke Math.~J.}} \textbf{92} (1998), 605--635.

\bibitem{dehority2020affinizations}
DeHority S., Affinizations of {L}orentzian {K}ac--{M}oody algebras and
 {H}ilbert schemes of points on {$K3$} surfaces, \href{https://arxiv.org/abs/2007.04953}{arXiv:2007.04953}.

\bibitem{goodman2009symmetry}
Goodman R., Wallach N.R., Symmetry, representations, and invariants,
 \textit{Graduate Texts in Mathematics}, Vol.~255, \href{https://doi.org/10.1007/978-0-387-79852-3}{Springer}, Dordrecht, 2009.

\bibitem{gyenge2017euler}
Gyenge \'A., N\'emethi A., Szendr\H{o}i B., Euler characteristics of {H}ilbert
 schemes of points on surfaces with simple singularities, \href{https://doi.org/10.1093/imrn/rnw139}{\textit{Int. Math.
 Res. Not.}} \textbf{2017} (2017), 4152--4159, \href{https://arxiv.org/abs/1512.06844}{arXiv:1512.06844}.

\bibitem{gyenge2018euler}
Gyenge \'A., N\'emethi A., Szendr\H{o}i B., Euler characteristics of {H}ilbert
 schemes of points on simple surface singularities, \href{https://doi.org/10.1007/s40879-018-0222-4}{\textit{Eur.~J. Math.}}
 \textbf{4} (2018), 439--524, \href{https://arxiv.org/abs/1512.06848}{arXiv:1512.06848}.

\bibitem{hartshorne1977algebraic}
Hartshorne R., Algebraic geometry, \textit{Graduate Texts in Mathematics},
 Vol.~52, \href{https://doi.org/10.1007/978-1-4757-3849-0}{Springer-Verlag}, New York~-- Heidelberg, 1977.

\bibitem{king1994moduli}
King A.D., Moduli of representations of finite-dimensional algebras,
 \href{https://doi.org/10.1093/qmath/45.4.515}{\textit{Quart.~J. Math. Oxford}} \textbf{45} (1994), 515--530.

\bibitem{kuznetsov2007quiver}
Kuznetsov A., Quiver varieties and {H}ilbert schemes, \href{https://doi.org/10.17323/1609-4514-2007-7-4-673-697}{\textit{Mosc. Math.~J.}}
 \textbf{7} (2007), 673--697, \href{https://arxiv.org/abs/math.AG/0111092}{arXiv:math.AG/0111092}.

\bibitem{LeuschkeWiegand12}
Leuschke G.J., Wiegand R., Cohen--{M}acaulay representations,
 \textit{Mathematical Surveys and Monographs}, Vol.~181, \href{https://doi.org/10.1090/surv/181}{Amer. Math. Soc.},
 Providence, RI, 2012.

\bibitem{McKay80}
McKay J., Graphs, singularities, and finite groups, in The {S}anta {C}ruz
 {C}onference on {F}inite {G}roups ({U}niv. {C}alifornia, {S}anta {C}ruz,
 {C}alif., 1979), \textit{Proc. Sympos. Pure Math.}, Vol.~37, Amer. Math.
 Soc., Providence, R.I., 1980, 183--186.

\bibitem{nakajima1994instantons}
Nakajima H., Instantons on {ALE} spaces, quiver varieties, and {K}ac--{M}oody
 algebras, \href{https://doi.org/10.1215/S0012-7094-94-07613-8}{\textit{Duke Math.~J.}} \textbf{76} (1994), 365--416.

\bibitem{Nakajima98}
Nakajima H., Quiver varieties and {K}ac--{M}oody algebras, \href{https://doi.org/10.1215/S0012-7094-98-09120-7}{\textit{Duke
 Math.~J.}} \textbf{91} (1998), 515--560.

\bibitem{nakajima2002geometric}
Nakajima H., Geometric construction of representations of affine algebras, in
 Proceedings of the {I}nternational {C}ongress of {M}athematicians, {V}ol.~{I}
 ({B}eijing, 2002), Higher Ed. Press, Beijing, 2002, 423--438.

\bibitem{nakajima2009quiver}
Nakajima H., Quiver varieties and branching, \href{https://doi.org/10.3842/SIGMA.2009.003}{\textit{SIGMA}} \textbf{5} (2009),
 003, 37~pages, \href{https://arxiv.org/abs/0809.2605}{arXiv:0809.2605}.

\bibitem{reiten1989two}
Reiten I., Van~den Bergh M., Two-dimensional tame and maximal orders of finite
 representation type, \href{https://doi.org/10.1090/memo/0408}{\textit{Mem. Amer. Math. Soc.}} \textbf{80} (1989),
 viii+72~pages.

\bibitem{Rosenfeld}
Rosenfeld B., Geometry of {L}ie groups, \textit{Mathematics and its
 Applications}, Vol.~393, \href{https://doi.org/10.1007/978-1-4757-5325-7}{Kluwer Academic Publishers Group}, Dordrecht, 1997.

\bibitem{savage2011quiver}
Savage A., Tingley P., Quiver {G}rassmannians, quiver varieties and the
 preprojective algebra, \href{https://doi.org/10.2140/pjm.2011.251.393}{\textit{Pacific~J. Math.}} \textbf{251} (2011),
 393--429, \href{https://arxiv.org/abs/0909.3746}{arXiv:0909.3746}.

\end{thebibliography}
\end{document}